\documentclass[11pt]{article}

\usepackage{biblatex}
\usepackage{amssymb,amsfonts,amsmath,amsthm}
\usepackage{geometry}
\usepackage{color,enumerate,graphicx,hyperref}
\usepackage{stix}
\usepackage[normalem]{ulem}
\usepackage{tikz}
\usetikzlibrary{hobby}

\newcommand{\R}{\mathbb R}

\newcommand{\Sphere}{\mathbb S}
\newcommand{\Z}{\mathbb Z}

\newcommand{\m}{m}

\newcommand{\intd}{\, \mathrm{d}}
\newcommand{\bp}{\Phi}

\newcommand{\abs}[1]{\lvert #1 \rvert}

\newtheorem{theorem}{Theorem}
\newtheorem{proposition}[theorem]{Proposition}
\newtheorem{definition}[theorem]{Definition}
\newtheorem{lemma}[theorem]{Lemma}

\newtheorem{remark}[theorem]{Remark}

\numberwithin{equation}{section}
\numberwithin{theorem}{section}

\renewbibmacro{in:}{%
	\ifentrytype{article}{}{\printtext{\bibstring{in}\intitlepunct}}}
\usepackage{todonotes}

\bibliography{bibliography.bib}
\title{Magnetic skyrmions under confinement}

\author{Antonin Monteil\footnote{Laboratoire d'Analyse et de
    Mathématiques Appliquées, Université Paris-Est - Créteil
    Val-de-Marne, Bâtiment P, 94010 Créteil, France} \footnotemark[5]
  \and Cyrill B. Muratov\footnote{Department of Mathematical Sciences,
    New Jersey Institute of Technology, Newark, NJ 07102, USA}
  \footnote{Dipartimento di Matematica, Universit\`a di Pisa, Largo
    Bruno Pontecorvo 5, 56127 Pisa, Italy} \and Theresa
  M. Simon\footnote{Institut f\"{u}r Analysis und Numerik,
    Universit\"{a}t M{\"u}nster, 48149 M{\"u}nster, Germany} \and
  Valeriy V. Slastikov\footnote{School of Mathematics, University of
    Bristol, Bristol BS8 1UG, United Kingdom}}

\date{\today}

\begin{document}
\maketitle

\begin{abstract}
  We present a variational treatment of confined magnetic skyrmions in
  a minimal micromagnetic model of ultrathin ferromagnetic films with
  interfacial Dzylashinksii-Moriya interaction (DMI) in competition
  with the exchange energy, with a possible addition of perpendicular
  magnetic anisotropy. Under Dirichlet boundary conditions that are
  motivated by the asymptotic treatment of the stray field energy in
  the thin film limit we prove existence of topologically non-trivial
  energy minimizers that concentrate on points in the domain as the
  DMI strength parameter tends to zero. Furthermore, we derive the
  leading order non-trivial term in the $\Gamma$-expansion of the
  energy in the limit of vanishing DMI strength that allows us to
  completely characterize the limiting magnetization profiles and
  interpret them as particle-like states whose radius and position are
  determined by minimizing a renormalized energy functional.  In
  particular, we show that in our setting the skyrmions are strongly
  repelled from the domain boundaries, which imparts them with
  stability that is highly desirable for applications. We provide
  explicit calculations of the renormalized energy for a number of
  basic domain geometries.
\end{abstract}

\tableofcontents

\section{Introduction}
\label{sec:intro}

Magnetic skyrmions are particle-like non-collinear spin textures that
were predicted to exist in non-centrosymmetric ferromagnets some 30
years ago \cite{bogdanov89,bogdanov89a,bogdanov94} and have been
recently observed in a number of magnetic systems
\cite{muhlbauer09,yu10,heinze11,romming13}. These coherent spin states
are enabled by the non-trivial topological characteristics of their
magnetizations \cite{nagaosa13,fert17} which endow them with a
considerable degree of thermal stability down to nanoscale and permit
observation of magnetic skyrmions at room temperature
\cite{moreau-luchaire16,boulle16,woo16}. The latter property makes
magnetic skyrmions attractive candidates as information carriers in a
new generation of spintronic devices for information technology
\cite{fert17,zhang20}.

In ultrathin ferromagnetic films exhibiting skyrmions, the
magnetization of the material may be described as a map from a
two-dimensional plane to a three-dimensional sphere at the level of
the continuum \cite{landau8,rohart13}. As skyrmions are particle-like
localized perturbations of the uniform ferromagnetic state, they must
belong to a homotopy class of the equivalent (after a stereographic
projection) continuous maps from $\Sphere^2$ to itself. These classes
are characterized by an integer topological degree, and the observed
magnetic skyrmion configurations display the degree $+1$ of the
identity map from $\Sphere^2$ to $\Sphere^2$
\cite{nagaosa13}.\footnote{Note a sign error in the computation of the
  skyrmion number in this reference.} Mathematically, these
configurations may be viewed as local minimizers of a suitable
micromagnetic energy functional among configurations within the above
homotopy class, and their existence was established for several models
\cite{melcher14,bms:prb20,bernand2019quantitative,bms:scipost22}.

In a minimal model relevant to ultrathin ferromagnetic films capped
with a layer of a heavy metal \cite{rohart13}, the energy consists of
a sum of the exchange energy forcing the magnetization to be constant
in space, the interfacial Dzyaloshinskii-Moriya interaction (DMI) that
promotes rotation of the magnetization vector, as well as the
perpendicular magnetic anisotropy that forces the magnetization to
align normally to the film plane and/or the Zeeman energy associated
with the perpendicular applied magnetic field that has the same effect
(for technical details, see section \ref{sec:main}). Note that the
problem of existence above is closely related to the one studied by
Lin and Yang in a two-dimensional Skyrme model
\cite{lin2004existenceofenergy,lin2004existence}. Bernand-Mantel et
al. established the asymptotic behavior of skyrmion solutions in the
case of vanishingly small DMI strength and demonstrated that in this
limit the magnetization profiles are close to the shrinking
Belavin-Polyakov profiles, i.e., the degree +1 harmonic maps from
$\R^2$ to $\Sphere^2$ \cite{belavin1975metastable}, which in the
minimal model described above are of N\'eel type
\cite{bms:prb20,bernand2019quantitative}.

In the physics literature, the inability to continuously deform a
topologically non-trivial skyrmion configuration into the
topologically trivial uniform ferromagnetic state is often referred to
as {\em topological protection} of magnetic skyrmions
\cite{nagaosa13}. We note that this is somewhat of a misnomer, as a
topologically non-trivial skyrmion configuration may in fact be
deformed discontinuously into the uniform ferromagnetic state via core
collapse by crossing a finite energy barrier \cite{bms22}. In
contrast, in finite samples such as nanodots or nanostrips that are of
particular interest to applications, there is strictly speaking no
topological obstruction that prohibits a homotopy between a skyrmion
solution and the trivial solution for example by moving the skyrmion
``through'' the boundary.  Nevertheless, these two solutions may still
be separated by an energy barrier, and the question of existence of
skyrmion solutions becomes more subtle.

In the minimal micromagnetic model that includes the stray field
effect only via an effective anisotropy term \cite{winter61}, Rohart
and Thiaville numerically constructed the N\'eel type
radially-symmetric skyrmion solutions in a circular nano-dot
\cite{rohart13}. It is unclear, however, whether these solutions
always represent local energy minimizers, as the exchange energy in
such a solution may be continuously lowered by moving the skyrmion
towards the domain boundary, breaking the radial symmetry of the
solution (see also section \ref{sec:free}). Numerical studies of the
minimal model in confined geometries do indicate the presence of a
finite energy barrier towards skyrmion disappearance through the
boundary under certain conditions
\cite{cortes-ortuno17,cortes-ortuno19,riveros21}. The solutions in
nanodisks were further analyzed numerically within the full
micromagnetic model that includes the non-local stray field effects
\cite{sampaio13,tejo18,aranda18a}. In particular, the obtained
numerical profiles exhibit a strong perpendicular alignment of the
magnetization at the domain edges, which can be explained by an
additional contribution of the stray field enhancing the perpendicular
magnetic anisotropy there (see also the experimental observations in
\cite{ho19}). As was shown in \cite{dms22}, in suitable thin film
limits for the considered class of materials the effect of the stray
field may be asymptotically accounted for via an effective penalty
term forcing the magnetization to align with the normal to the film
plane at the domain boundary, similarly to what happens in other
ferromagnetic thin film problems \cite{kohn05arma}. In our problem,
this should lead to the skyrmion being repelled from the sample edges.

In view of the above arguments, it is physically reasonable to
consider the situation in which the magnetization at the film edge is
rigidly aligned with a normal to the film plane. This may either be
achieved via sending the penalization of the deviations at the
boundary to infinity (corresponding to an appropriate choice of the
material and geometric parameters \cite{dms22}), or it could be the
result of patterning the substrate of an extended ferromagnetic film
with a strongly magnetically anchoring material (for a related
approach, see \cite{ohara21}). Using these Dirichlet boundary
conditions restores the possibility of topological protection, as
continuous maps from a bounded two-dimensional domain to $\Sphere^2$
with the boundary values pinned to a single direction can once again
be classified by their topological degree. However, it is still not a
priori clear whether minimizers would be attained in such a setting,
as the possibility of a skyrmion shrinking to a point and collapsing
is not excluded.

In this paper, we present a variational treatment of the minimal
micromagnetic model of confined magnetic skyrmions in ultrathin
ferromagnetic films with interfacial DMI, in which the confinement is
provided by the Dirichlet boundary condition that forces the
magnetization to take one direction normal to the film plane at the
two-dimensional domain edge. We first prove existence of degree $+1$
minimizers of the energy consisting of the sum of the exchange and the
interfacial DMI terms (with a possible addition of the perpendicular
magnetic anisotropy term). We then focus on the regime in which the
DMI is a perturbation to the exchange energy and develop a
$\Gamma$-expansion of the energy in the limit of vanishing DMI
strength. This leads to the appearance of a {\em renormalized energy}
which determines asymptotically both the location and the radius of
the skyrmion, whose shape is shown to be close to a N\'eel type degree
+1 harmonic map from $\R^2$ to $\Sphere^2$. Lastly, we explicitly
construct the minimizers of the renormalized energy in the case of
disk and strip domains. In particular, we show that the energy
minimizing skyrmions are located in the disk center and on the strip
midline, respectively, due to the effective repulsive interaction
provided by the excess exchange energy from the tail of the
magnetization profile. This confirms the physical expectation based on
the numerical simulations that skyrmions can be robust particle-like
objects even in finite samples of varying geometry.

\subsection{Informal discussion of the results}

From a mathematical standpoint, the confinement provided by the
boundary data in fact simplifies the proof of existence of skyrmions
compared to the case of the whole plane, since the translational
symmetry of the problem is broken.  In order to obtain the parameters
describing the asymptotic behavior, we apply the rigidity of degree
$\pm 1$ harmonic maps from $\R^2$ to $\Sphere^2$ obtained by
Bernand-Mantel, Muratov and Simon \cite{bernand2019quantitative} after
extending the magnetizations by a constant outside the domain using
the Dirichlet boundary condition.  This allows us to define the
location, radius and rotation angle of the skyrmion.  As is common in
$\Gamma$-convergence arguments, we first obtain qualitative
information such as linear scaling of the radius in the DMI constant
or the fact that skyrmions are repelled from the boundary via
non-optimal estimates, in order to obtain compactness properties of
the energies.

A finer analysis requires us to also keep track of the tail correction
to the skyrmion profile necessary to enforce the boundary condition.
Here, the skyrmion position interacts with the boundary through the
solution of the linearization of the harmonic map problem at the
constant state given by the boundary condition, i.e., Laplace's
equation for the in-plane components.  A correction to the skyrmion
core is not necessary to first order as the Belavin-Polyakov profiles
are the exact degree one minimizers of the Dirichlet energy on the
whole plane.  The location of the optimal skyrmion is set by
minimizing the interaction with the boundary, and the N{\'e}el
character of the profile arises via minimizing the DMI term among all
rotation angles.  The radius then optimizes the balance of the two
contributions.  In simple domains, such as balls and strips, the
Laplace's equation determining the tail correction can be explicitly
solved by means of complex analysis, thus giving the full solution of
the limiting problem in these cases.

Finally, we additionally include the perpendicular anisotropy at an
appropriate scaling in the DMI constant.  As the radius scales
linearly in the DMI constant and the anisotropy energy of an exact
Belavin-Polyakov profile is well-known to have a logarithmic
divergence in its tail
\cite{doring2017compactness,bernand2019quantitative}, we consider
effective anisotropies scaled down with the logarithm of the DMI
strength.  In this regime the anisotropy is essentially a continuous
perturbation of our original problem with respect to the topology we
determine the $\Gamma$-limit in.  Furthermore, due to the fact that it
is only the tail that contributes to the anisotropy at leading order
and that it is of logarithmic character, we obtain that its
contribution in the limit is in fact independent of the shape of the
domain.

We note that the variational problem considered by us bears several
similarities with the one for the classical Ginzburg-Landau model
(without the magnetic field) \cite{bethuel,pacard}, in which the
boundary data with a non-trivial topological degree force minimizers
to form point-like vortices in the domain interior as the small
parameter of the model goes to zero. Our problem for the magnetic
skyrmion behavior in the limit of vanishing DMI strength thus provides
a micromagnetic setting for the celebrated questions of Matano for
Ginzburg-Landau vortices \cite{bethuel,struwe94}. Towards answering
this type of questions, we show that the skyrmion in a disk
concentrates at the disk center in the limit and explicitly compute
its asymptotic magnetization profile. We point out, however, that the
analysis of the limit micromagnetic problem is considerably more
delicate, as in contrast to the Ginzburg-Landau problem the energy of
a single skyrmion remains finite in the limit, and, therefore the tail
contribution of the Dirichlet energy does not decouple from the
problem for the skyrmion core. In particular, contrary to the
Ginzburg-Landau vortex problem \cite{mironescu96,shafrir95}, the
radius of the skyrmion turns out to be affected by the shape of the
domain through the solution of the limit problem in the tail.

As in the problem of Ginzburg-Landau vortices \cite{bethuel,pacard},
it is also natural to ask whether multiple skyrmion configurations may
be similarly described in the vanishing DMI strength limit. In fact,
the micromagnetic energy has been shown numerically to exhibit a
multitude of local energy minimizers other than a single magnetic
skyrmion \cite{rybakov19,kuchkin20}. However, our present analysis
does not easily extend to the case of magnetization configurations of
degree other than $+ 1$. Even at the level of existence we cannot rule
out the collapse of minimizing sequences, failing to yield minimizers
with a prescribed degree in this case. For the limit behavior of
vanishing DMI strength, we also no longer have the quantitative
rigidity estimate for the harmonic maps of arbitrary degree, which is
the key tool in our analysis of a single skyrmion
\cite{bernand2019quantitative}. In fact, such an estimate has been
recently shown to be false for degree 2 harmonic maps
\cite{deng20}. Similarly, we cannot give a positive answer to the
existence of anti-skyrmions, i.e., minimizers among configurations
with degree $-1$, as we do not know whether the basic energy bound in
Lemma \ref{lemma:construction} holds in this class.

\subsection{Outline of the paper}
This paper is organized as follows. In section \ref{sec:main}, we give
the precise definitions of the micromagnetic energy, admissible
classes, and the limit processes under consideration and then formulate
our main results. In section \ref{sec:exist}, we prove existence of
minimizers in the considered non-trivial topological class of maps
with degree +1. In section \ref{sec:Gamma}, we derive the first-order
term in the $\Gamma$-expansion of the energy in the DMI strength
beyond the classical topological lower bound at zeroth order. Then, in
section \ref{sec:lim} we explicitly compute the renormalized energy
for a number of geometries. Finally, in section \ref{sec:anicont} we
show how to include the perpendicular magnetic anisotropy as a
continuous perturbation to the limit energy.

\paragraph{Acknowledgements.}

The work of C.B.M. was supported, in part, by NSF via grant
DMS-1908709.  A.M. and V.V.S. acknowledge support by Leverhulme grant
RPG-2018-438.  The work of T.M.S. is funded by the Deutsche
Forschungsgemeinschaft (DFG, German Research Foundation) under
Germany's Excellence Strategy EXC 2044 –390685587, Mathematics
Münster: Dynamics – Geometry – Structure, and by the Deutsche
Forschungsgemeinschaft (DFG, German Research Foundation) - Project-ID
211504053 - SFB 1060.

\section{Main results}
\label{sec:main}

\subsection{Definition of the energy}
On a bounded domain $\Omega \subset \R^2$ with Lipschitz
  boundary, we consider the set of admissible functions
\begin{align}
  \mathcal A = \left\{ \m \in H^1(\Omega; \Sphere^2) : \ \m= -
  e_3 \hbox{ on } \partial \Omega , \ \mathcal{N}(\m) =1 \right\}, 
\end{align}
where the degree of a function $\m \in \mathring H^1(\R^2;\Sphere^2)$
is defined as
\begin{align}
  \label{eq:N}
  \mathcal N (\m) =\frac{1}{4 \pi} \int_\Omega \m \cdot (\partial_1 \m
  \times \partial_2 \m) \intd x, 
\end{align}
and we extend $m\in \mathcal{A}$ to the whole of $\R^2$ by setting
$m = -e_3$ outside $\Omega$. Here, as usual, we define
\begin{align}
    \mathring H^1(\R^2, \Sphere^2) := \left\{m \in H^1_{\mathrm{loc}}
  (\R^2; \R^3) :
  \int_{\R^2} |\nabla m|^2 \intd x < \infty, \ | m | = 1 \
  \text{a.e. in } \R^2 \right\}. 
\end{align}
It is well known that $\mathcal N(\m) \in \Z$
for any $\m \in \mathring H^1(\R^2;\Sphere^2)$, see Brezis and Coron
\cite{brezis1983large}.  For $\m \in \mathcal{A}$ we wish to minimize
the energy
\begin{align}\label{def:energy}
  \mathcal{E}_\kappa(\m) = \int_{\Omega} \left( |\nabla \m|^2 -2 \kappa 
  \m'\cdot \nabla m_3   \right) \intd x, 
\end{align}
where $\kappa \in \R$ is the DMI constant and we use the convention
$m = (m',m_3)$, with $m'$ taking values in $\R^2$.  Passing from $\m$
to $\tilde \m := (-m', m_3)$ when minimizing $\mathcal E_\kappa$ in
the case of $\kappa<0$, throughout the rest of the paper we may assume
that $\kappa \geq 0$.

\subsection{Statement of the results}

We first make sure that the energy indeed admits minimizers. Due to
the Dirichlet boundary conditions, minimizers exist for all $\kappa>0$
sufficiently small even in the absence of the anisotropy penalizing
the out-of-plane component of the magnetization. Note that at the same
time the infimum of the energy is not attained for $\kappa = 0$ (see
below).
  
\begin{theorem}\label{prop:existence}
  There exists $\kappa_0 > 0$ depending only on $\Omega$ such that for
  all $0 < \kappa < \kappa_0$ there exists a minimizer of
  $\mathcal{E}_\kappa$ over $\mathcal{A}$.
\end{theorem}

More importantly, we are also able to give a precise description of
the minimizers for $\kappa$ being small, i.e., the parameter regime in
which one does have skyrmions.  In particular, we can express their
location and radius in terms of an optimization over the tail of the
skyrmion.

To make this statement precise, we define the standard
Belavin-Polyakov profile
\begin{align}\label{def:BPprofile}
  \bp (x):= \begin{pmatrix}
    - \frac{2x}{1+|x|^2} , & \frac{1-|x|^2}{1+|x|^2}
	\end{pmatrix}
\end{align}
for $x\in \R^2$, which is the negative of the inverse stereographic
projection, and we denote the set of all Belavin-Polyakov profiles by
\begin{align}\label{def:bp_profiles}
  \mathcal{B} := \left\{R \bp \left( \rho^{-1} (\cdot - {
  a})\right): R \in SO(3), \ \rho>0, \  { a} \in \R^2 \right\}. 
\end{align}
They arise as configurations achieving equality in the sharp
topological bound
\begin{align}\label{eq:topological_bound}
  \int_{\R^2} |\nabla \m|^2 \intd x \geq 8\pi |\mathcal{N}(\m)|
\end{align}
with degree $\mathcal{N}=1$.  In particular, they are precisely the
minimizing harmonic maps of degree one, see Belavin and Polyakov
\cite{belavin1975metastable} and \cite[Lemma A.1]{brezis1983large}.
It is therefore not surprising and indeed well known
\cite{doring2017compactness,bernand2019quantitative}, that minimizers
of micromagnetic-type energies augmented with DMI should approach the
set $\mathcal{B}$ when the Dirichlet energy dominates, i.e., when
$\kappa \ll 1$.  We can thus attempt to express the location and the
radius of the skyrmions as $ a \in \R^2$ and $\rho>0$ of an approached
Belavin-Polyakov profile in this regime. Notice that for $\kappa = 0$
an equality in \eqref{eq:topological_bound} is achieved by a sequence
of truncated Belavin-Polyakov profiles with vanishing radius, which
fails to converge to an element in $\mathcal A$. This statement
remains true also in the presence of an additional out-of-plane
anisotropy term (see section \ref{sec:aniso}).

However, as we expect the radius of the minimizers to shrink compared
to the size of the domain as $\kappa \to 0$, we can only expect the
close-by Belavin-Polyakov profiles to converge after a
rescaling. Consequently, we have to find a Belavin-Polyakov profile
for each minimizer at positive $\kappa$ in a controlled way.  An
appropriate set of tools for such a purpose has been identified by
Bernand-Mantel, Muratov, and Simon in the form of a quantitative
rigidity result for degree one harmonic maps:

\begin{theorem}[{\cite[Theorem
    2.4]{bernand2019quantitative}}]\label{thm:rigidity} 
  For $\m \in \mathcal{A}$, let the Dirichlet excess be
\begin{align}
	Z(\m) := \int_{\Omega} |\nabla \m|^2 \intd x - 8\pi
\end{align}
and the Dirichlet distance to the set of the Belavin-Polyakov profiles
be
\begin{align}
  D(m; \mathcal{B}) := \inf_{\mathbf{\phi}\in \mathcal{B}} \left(
  \int_{\R^2} | \nabla ( \m - \mathbf{\phi})|^2 \intd x
  \right)^\frac{1}{2},
\end{align}
where as usual $m$ is extended outside of $\Omega$ by $-e_3$.  Then
the infimum in the definition of $D(m; \mathcal{B})$ is achieved,
i.e., there exists a Belavin-Polyakov profile closest to each
$m \in \mathcal{A}$. Moreover, there exists a universal constant
$\eta>0$ such that for all $m\in \mathcal{A}$ we have
\begin{align}\label{eq:rigidity}
  \eta D^2 (\m;\mathcal{B}) \leq Z(\m).
\end{align}
\end{theorem}
Shorter, alternative proofs of this statement have later been provided
by Hirsch and Zemas \cite{hirsch2021note} and Topping
\cite{topping2021rigidity}.

In order to identify the Belavin-Polyakov profiles corresponding to
minimizers of $\mathcal{E}_\kappa$ in the limit $\kappa \to 0$,
we turn to computing the $\Gamma$-limit in a suitable topology
retaining the location, the radius, the global rotation and the
skyrmion tail.  To this end, we have to identify the correct higher
order $\Gamma$-expansion of the energy.  By roughly minimizing
over the above quantities, we will find in Lemma
\ref{lemma:construction} below that there exists a constant $C>0$
depending only on $\Omega$ such that
\begin{align}
	\inf_{\mathcal{A}} \mathcal{E}_\kappa \leq 8\pi - C\kappa^2.
\end{align}
This suggests to seek a $\Gamma$-limit of the functional
$\frac{\mathcal{E}_\kappa - 8\pi}{\kappa^2}$.

However, in order to rule out some behaviors of finite energy
sequences that minimizers will not exhibit, such as skyrmions
  shrinking too fast or their centers approaching the boundary of
  $\Omega$, we will restrict our attention to magnetizations whose
energy is sufficiently low, i.e. we restrict the admissible set
to
\begin{align}\label{def:admissible_restricted}
  \mathcal{A}_\kappa= \left\{ m\in \mathcal{A}:  \mathcal{E}_\kappa(m)
  - 8\pi <0 \right\}. 
\end{align}
Furthermore, in the $\Gamma$-limit we will only consider
magnetizations $m_\kappa \in \mathcal{A}_\kappa$ which satisfy
\begin{align}
	\liminf_{\kappa \to 0} \frac{\mathcal{E}_\kappa (m_\kappa) -
  8\pi}{\kappa^2} < 0 
\end{align}
Note that this corresponds to a finite energy sequence for the
functional $\frac{\kappa^2}{|\mathcal{E}_\kappa - 8\pi|}$ defined on
$\mathcal{A}_\kappa$.

To specify the topology for the $\Gamma$-limit, given
$m\in \mathcal{A}_\kappa$ we choose
$\phi_m(x):= R \Phi (\rho^{-1}(x-a))$ for $R \in SO(3)$,
$\rho >0$ and $a \in \R^2$ to minimize the Dirichlet distance to
$m$ after extension to $\R^2$ by $-e_3$.  In addition, we will also
consider the tail of the skyrmion $w_m := m+e_3 - \phi_m - Re_3$.
Guessing from the construction of Lemma \ref{lemma:construction}, we
expect $\rho \sim \kappa$ and
$\|\nabla w_m \|_{L^2(\R^2)} \sim \rho$.  It turns out that the
information $m= -e_3$ in $\R^2\setminus \Omega$ will translate into an
asymptotic expression for $ \rho^{-1} w_m$ outside $\Omega$, see
Lemma \ref{lemma:skyrmion_tail}.  This motivates the following:

\begin{definition}\label{def:topology}
  Let
\begin{align}\label{def:A0tilde}
  \begin{split}
    \widetilde{\mathcal{A}_0}& := \left\{ R_0 \in SO(3): R_0 e_3=e_3
    \right\} \times (0,\infty)\times \Omega.
  \end{split}
\end{align}
We then say that a sequence $m_{\kappa_n} \in \mathcal{A}_{\kappa_n}$
BP-converges to $(R_0,r_0, a_0) \in \widetilde{\mathcal{A}_0}$ as
$\kappa_n \to 0$ if and only if the following holds: There exist
$R_n \in SO(3)$, $\rho_n >0$, $a_n \in \Omega$ such that for
$\phi_n := R_n \Phi(\rho_n^{-1}(\bullet - a_n)) \in \mathcal{B}$
we have
\begin{align}
  \limsup_{n \to \infty } \kappa_n^{-2}\int_{\R^2} |\nabla
  (m_{\kappa_n} - \phi_n )|^2 \intd x
  &  <
    \infty,\label{eq:m_close_to_phi}\\ 
  R_0 & = \lim_{n \to \infty } R_n,\\
  r_0 & =  \lim_{n \to \infty } \frac{\rho_n}{\kappa_n}, \label{eq:radius_kappa}\\
  a_0 & =  \lim_{n \to \infty }  a_n.
\end{align} 
\end{definition}

\begin{remark}\label{rem:unique}
  By the first condition and the triangle inequality in
  $\mathring H^1(\R^2)$, one can see that BP-limits are unique.
\end{remark}

We are now in a position to give the $\Gamma$-limit of
$\frac{\mathcal{E}_\kappa - 8\pi}{\kappa^2}$ with respect to the above
convergence.

\begin{definition}
  \label{def:T}
  For $(R_0,r_0, a_0 ) \in \widetilde{\mathcal{A}_0}$ let 
  \begin{align}
    \label{eq:E0}
          \mathcal{E}_0 (R_0,r_0, a_0) := r_0^2 T( a_0) - 2 r_0
          \int_{\R^2} (R_0\Phi)' \cdot \nabla \Phi_3 \intd x, 
	\end{align}
	where the Dirichlet contribution of the tail correction is
	\begin{align}
          \label{eq:Ta0}
		T( a_0):=\inf\left\{\int_{\R^2} | \nabla u |^2 \intd x
          :  u(x) =2 \frac{x-a_0}{\abs{x-a_0}^2}\text{ in }
          \R^2\setminus\Omega\right\}. 
	\end{align}
	We furthermore define a restricted admissible set
	\begin{align}\label{def:A0}
          \mathcal{A}_0:= \left\{ (R_0,r_0, a_0) \in
          \widetilde{\mathcal{A}_0}: \mathcal{E}_0(R_0,r_0,
          a_0)<0\right\}. 
	\end{align}
\end{definition}\label{def:limit}
	
We can then state the $\Gamma$-convergence.
	
\begin{theorem}\label{thm:convergence}
  The $\Gamma$-limit as $\kappa\to 0$ of the functionals
  $\frac{\mathcal{E}_\kappa - 8\pi}{\kappa^2}$ restricted to
  $\mathcal{A}_\kappa$ with respect to the BP-convergence is given by
  $\mathcal{E}_0$ restricted to $\mathcal{A}_0$ in the sense that we
  have the following:
	\begin{enumerate}[(i)]
        \item For every sequence of $\kappa_n \to 0$ and
          $m_{\kappa_n} \in \mathcal{A}_{\kappa_n}$ with
          $\liminf_{n\to \infty }
          \frac{\mathcal{E}_{\kappa_n}(m_{\kappa_n}) -
            8\pi}{\kappa_n^2}<0$ there exists a subsequence (not
          relabeled) and $(R_0,r_0, a_0 ) \in \mathcal{A}_0$ such that
          $m_{\kappa_n}$ BP-converges to $(R_0,r_0, a_0)$.
        \item Let $\kappa_n \to 0$, let
          $m_{\kappa_n} \in \mathcal{A}_{\kappa_n}$ BP-converge to
          $(R_0,r_0, a_0 ) \in \mathcal{A}_0$ and let
		\begin{align}
                  \liminf_{n\to
                  \infty}\frac{\mathcal{E}_{\kappa_n}(m_{\kappa_n}) -
                  8\pi}{\kappa_n^2} < 0. 
		\end{align}
		Then we have
		\begin{align}
			\liminf_{n\to
                  \infty}\frac{\mathcal{E}_{\kappa_n}(m_{\kappa_n}) -
                  8\pi}{\kappa_n^2}  \geq \mathcal{E}_0 (R_0,r_0, a_0
                  ). 
		\end{align}
              \item For every $(R_0,r_0, a_0) \in \mathcal{A}_0$ and
                every sequence of $\kappa_n \to 0$ there
                exist
                $m_{\kappa_n} \in \mathcal{A}_{\kappa_n}$
                BP-converging to $(R_0,r_0, a_0)$ such that
		\begin{align}
                  \limsup_{n\to \infty}
                  \frac{\mathcal{E}_{\kappa_n}(m_{\kappa_n}) -
                  8\pi}{\kappa_n^2}  \leq \mathcal{E}_0 (R_0,r_0, a_0). 
		\end{align}
	\end{enumerate}
\end{theorem}

\begin{remark}
  The above version of $\Gamma$-convergence is equivalent to the usual
  notion for the functionals
  $\frac{\kappa^2}{|\mathcal{E}_\kappa - 8\pi|}$ and
  $|\mathcal{E}_0|^{-1}$ restricted to $\mathcal{A}_\kappa$ and
    $\mathcal{A}_0$, respectively.
\end{remark}

Notice that the last term in the definition of $\mathcal E_0$ is
clearly minimized by $R_0 = \mathrm{id}$ among all $R_0$ satisfying
$R_0 e_3 = e_3$, since this achieves an absolute maximum of the
integrand by pointwise Cauchy-Schwarz inequality in view of the fact
that $\Phi'$ is collinear to $\nabla \Phi_3$. Thus from the fact that
$\int_{\R^2} \Phi' \cdot \nabla \Phi_3 \intd x = 4 \pi$, see
\cite[Lemma A.5]{bernand2019quantitative}, we have
  \begin{align}
    \label{eq:E0lower}
    \mathcal E_0(R_0, r_0, a_0) \geq \mathcal E_0(\mathrm{id}, r_0,
    a_0) =  T(a_0) \left( r_0 - {4 \pi \over T(a_0)} \right)^2  -{16
    \pi^2 \over T(a_0)}.   
  \end{align}
  Upon minimizing $\mathcal E_0$ over $\widetilde{\mathcal A_0}$, we
  can saturate the lower bound in \eqref{eq:E0lower} and obtain the
  following characterization of the minimizers of $\mathcal
  E_\kappa$.

\begin{theorem}\label{thm:minimizers}
  Let $\kappa_n \to 0$ as $n \to \infty$ and let $m_{\kappa_n}$ be
  minimizers of $\mathcal{E}_{\kappa_n}$ over $\mathcal A$.  Then
  there exists a subsequence (not relabeled) and
  $ a_0 \in \operatorname{arg min}_{a\in \Omega} T(a)$ such that with
	\begin{align}
          r_0& := \frac{4\pi}{T( a_0)},\\
          R_0 & := \operatorname{id}
	\end{align}
	we get for
        $\phi_n := \Phi\left( \bullet \ - \ a_0 \over r_0 \kappa_n
        \right) \in \mathcal{B}$ and all $n \in \mathbb{N}$ that
	\begin{align}\label{eq:rate_of_convergence}
          \int_{\R^2} |\nabla (m_{\kappa_n} - \phi_n)|^2 \intd x \leq
          C \kappa_n^2  
	\end{align}
	and
	\begin{align}
          \lim_{n \to \infty} \frac{\mathcal E_{\kappa_n} (m_{\kappa_n})
          - 8 \pi}{
          \kappa_n^2}
          = -  \frac{16\pi^2}{T(a_0)}. 
	\end{align}
      \end{theorem}

      \noindent In particular, this theorem says that as
      $\kappa_n \to 0$ the appropriately translated and dilated
      minimizer $m_{\kappa_n}( r_0 \kappa_n( \bullet ) + a_0 )$
      converges to the canonical Belavin-Polyakov profile $\Phi$ in
      Dirichlet distance, up to a subsequence. In the original
      variables, the energy minimizing profile is, therefore, close to
      the Belavin-Polyakov profile of N\'eel type centered at $a_0$
      and with the small radius $\rho_n = r_0\kappa_n$.

\subsection{Explicit minimizers for specific domains}

We next give several examples of geometries, in which an explicit
minimizer of the limit problem may be obtained. We use the standard
identification of the complex plane with $\mathbb R^2$ and write
$z \in \mathbb C$ to denote a vector in the plane.  The symbol
$\bar z$ denotes the complex conjugate of $z$.  We also introduce the
Wirtinger derivatives
$\partial_z = \frac12 (\partial_x - i \partial_y)$ and
$\partial_{\bar z} = \frac12 (\partial_x + i \partial_y)$, acting on
$u : \mathbb C \to \mathbb C$.

Clearly, the infimum in \eqref{eq:Ta0} is attained by the unique
  harmonic extension of $u$ from $\partial \Omega$ into $\Omega$. With
  $u_{z_0} : \mathbb C \to \mathbb C$ solving
  \begin{align}
    \label{eq:ez0}
    \Delta u_{z_0} = 0 \ \text{in} \ \Omega, \qquad
  u_{z_0}(z) = {2 \over \bar z -  
  \bar z_0} \ \text{in} \ \mathbb C \setminus \Omega,
  \end{align}
  one can then write the limit energy associated with a skyrmion
  centered at $z_0 \in \mathbb C$ as
\begin{align}
  \label{eq:Tuz}
  T(z_0) = \int_{\R^2} \nabla \bar u_{z_0} \cdot \nabla u_{z_0} \,
  \intd x.
\end{align}
The following proposition allows us to reduce the computation of
$T(z_0)$ to evaluating a derivative of $u_{z_0}(z)$ at $z = z_0$ for
the considered geometries.
  
\begin{proposition}\label{lemma:representation}
  Let $\Omega \subset \mathbb{C}$ be a simply connected bounded domain
  with a boundary of class $C^{1,\alpha}$, for some
  $\alpha \in (0,1)$.  Then we have
  \begin{align}
    \label{eq:TzW}
    T(z_0) = 8\pi \partial_z u_{z_0}(z_0).
  \end{align}
\end{proposition}
 
We note that due to the continuous dependence of the boundary values
of $u$ on $z_0$, the function $T(z_0)$ is continuous for all
$z_0 \in \Omega$. Moreover, since $|\nabla u (x)|$ behaves like
$|x-z_0|^{-2}$ for all $x\in \R^2\setminus \Omega$ and all competitors
$u$ from equation \eqref{eq:Ta0}, which is not $L^2$-integrable on
$\R^2\setminus \Omega$ if $a_0=z_0\in \partial \Omega$, we have
$T(z_0) \to +\infty$ as $z_0$ approaches $\partial \Omega$. Therefore,
$T(z_0)$ always attains its minimum for some $z_0 \in \Omega$.

\subsubsection{Disks}

For the special choice $\Omega = B_\ell(0)$ with $\ell > 0$
we can fully solve the above minimization problem, obtaining that the
skyrmion will be located in the disk's center.

\begin{proposition}\label{prop:balls}
  For $\Omega = B_\ell(0)$ and $z_0\in \Omega$, the map
  achieving $T(z_0)$ is given by
	\begin{align}\label{def:minimizer_ball}
          u_{z_0}(z) = \begin{cases}
            \frac{2z }{\ell^2 - \bar{z}_0 z }  & \text{ if } z
            \in B_{\ell}(0), \\ 
            \frac{2}{\bar z- \bar{z}_0} & \text{ if } z
            \in
            \mathbb C \setminus B_{\ell}(0).
		\end{cases}
	\end{align}
        Its energy is given by
          \begin{align}
            T( z_0) = \frac{16 \pi\ell^2}{ (\ell^2-|z_0|^2)^2},  
          \end{align}
          which is minimized by $z_0 = 0$ with
          $T(0)= \frac{16\pi}{ \ell^2}$. The rescaled skyrmion radius
          is $r_0= \frac{\ell^2}{4}$ and the corresponding limiting
          energy is
          $\mathcal{E}_0 \left( \operatorname{id}, \frac{\ell^2}{4}, 0
          \right) = - \pi \ell^2$.
\end{proposition}

Note that the minimizer achieving $T(a_0)$ has the special property of
being a holomorphic function in $\Omega$.  In the next example of
strips we will see that this does not necessarily have to be the case.

\subsubsection{Strips}

We can also consider the energy \eqref{def:energy} on strips
$\Omega_\ell = \mathbb{R} \times (-\ell/2,\ell/2)$ for $\ell >0$.
Technically, the previous statements do not apply as $\Omega_\ell$ is
not bounded.  However, the arguments can be adjusted straightforwardly
as strips support Poincar{\'e} inequalities. We will give the
modifications in section \ref{sec:proof_stripes} below.

The only change in the resulting statement is that in the
BP-convergence we will, due to the translational invariance of
$\Omega_\ell$ in the first component, only track the second component
of the skyrmion center, so that the limiting set is
\begin{align}
  \widetilde{\mathcal{A}_0}
  & := \left\{ R_0 \in SO(3): R_0 e_3=e_3
    \right\} \times (0,\infty)\times
    \left(-\frac{\ell}{2},\frac{\ell}{2}\right). 
\end{align}
Furthermore, the limiting energy is given by
\begin{align}
  \mathcal{E}_0 (R_0,r_0, y_0) := r_0^2 T( i y_0) - 2 r_0
  \int_{\R^2} (R_0\Phi)' \cdot \nabla \Phi_3 \intd x, 
\end{align}
where
\begin{align}
  T( i y_0 ):=\inf\left\{ \int_{\R^2}  | \nabla u |^2 \intd x  :  
  u(z) =\frac{2}{\bar z + i y_0} \text{ in }
  \mathbb C \setminus\Omega_\ell  \right\}. 
\end{align}

Also this problem can be solved explicitly.
\begin{proposition}\label{prop:stripes}
  For $\ell>0$, $\Omega_\ell = \mathbb{R} \times (-\ell/2,\ell/2)$,
  and $y_0 \in (-\ell/2,\ell/2)$, the map achieving
  $T(i y_0)$ is given by
	\begin{align}\label{def:minimizer_stripes}
          u_{y_0}(z) = \begin{cases}
            \frac{\pi}{\ell} \tanh\left( \frac{\pi}{2\ell} (z+i
              y_0)  \right) - \frac{\pi}{\ell} \coth \left(
              \frac{\pi}{2\ell} (\bar z  + i y_0)  \right)+  \frac{2
            }{\bar z + i y_0 }  & \text{ if } z \in \Omega_\ell, \\ 
            \frac{2}{\bar z + i y_0} & \text{ if } z \in
            \mathbb C \setminus
            \Omega_\ell.
		\end{cases}
	\end{align}
        Its energy is given by
        \begin{align}
          T(iy_0) = \frac{4\pi^3}{\ell^2
          \cos^2\left(\frac{\pi y_0}{\ell} \right) },  
        \end{align}
        which is minimized by $y_0 = 0$ with
        $T(0)= \frac{4\pi^3 }{\ell^2}$. The rescaled skyrmion radius
        is $r_0= \frac{\ell^2}{\pi^2}$ and the corresponding limiting
        energy is
        $\mathcal{E}_0 \left( \operatorname{id}, \frac{\ell^2}{\pi^2},
          0 \right) = - \frac{4 \ell^2}{\pi}$.
      \end{proposition}

      The formula for $u_{y_0}$ above was obtained by computing the
      harmonic extension of the boundary data in Fourier space, but to
      verify its validity we only need to check that it satisfies the
      conditions defining $u_{y_0}$.

\subsubsection{Half-plane}
\label{sec:halfplane}

For the half-space $\Omega = \R \times (-\infty, 0)$ our rigorous
arguments cannot be salvaged, and indeed in this case the energy can
be easily seen to be unbounded from below.  However, we may still
consider the problem as arising from a limiting procedure where the
distance of the skyrmion center to the boundary of growing, smooth
domains is fixed.  Then we obtain the problem
\begin{align}
  \mathcal{E}_0 (R_0,r_0, y_0) := r_0^2 T( i y_0 ) - 2 r_0
  \int_{\R^2} (R_0\Phi)' \cdot \nabla \Phi_3 \intd x, 
\end{align}
defined on
\begin{align}
  \widetilde{\mathcal{A}_0}& := \left\{ R_0 \in SO(3): R_0 e_3=e_3
                             \right\} \times (0,\infty)\times
                             \left(-\infty, 0\right) 
\end{align}
and where
\begin{align}
  T(iy_0):=\inf\left\{ \int_{\R^2} | \nabla u |^2 \intd x  :
  u(z) = \frac{2}{\bar z + i y_0}  \text{ in }
  \mathbb C\setminus\Omega\right\},
\end{align}
where the skyrmion is located at $z_0 = i y_0$ with $y_0 < 0$ in the
limit. 

The straightforward solution then gives information about how the
energy of the skyrmion behaves as it approaches the boundary.  Of
course, the repelling effect of the boundary can also be seen from our
rigorous analysis.  However, in this situation, the estimate is
especially transparent.

\begin{proposition}\label{prop:half_space}
  For $\Omega = \mathbb{R} \times (-\infty,0)$, and
  $y_0 \in (-\infty, 0 )$, the map achieving $T(i y_0)$
  is given by
	\begin{align}\label{def:minimizer_half_space}
          u_{y_0}(z) = \begin{cases}
            \frac{2}{z+ i y_0}  & \text{ if } z \in \Omega, \\
            \frac{2}{\bar z+ i y_0} & \text{ if } z \in
            \mathbb C \setminus
            \Omega.
		\end{cases}
	\end{align}
        Its energy is given by $T(i y_0) = \frac{4\pi}{y_0^2}$, the
        corresponding rescaled skyrmion radius is $r_0= y_0^2$ and the
        limiting energy is
        $\mathcal{E}_0 \left( \operatorname{id}, y_0^2 , y_0 \right) =
        - 4 \pi y_0^2$.
      \end{proposition}

\subsection{Adding anisotropy}
\label{sec:aniso}

We may also consider the case where we augment our energy by an
anisotropy term. In order not to significantly change the behavior of
the $\Gamma$-limit, we choose the quality factor $Q$ in dependence of
$\kappa$ such that, after the renormalization in the thin film limit
by the local contribution of the stray field term, the resulting term
is essentially a compact perturbation of our above results.  As it is
well known that the anisotropy contribution of a skyrmion with radius
$\rho>0$ behaves like $\rho^2 |\log \rho|$, see for example
\cite{doring2017compactness,bernand2019quantitative}, and as in our
case $\rho \sim \kappa$, the appropriate scaling is
$Q - 1 = \lambda |\log \kappa|^{-1}$ for some $\lambda >0$.
Consequently, we obtain the modified energy
\begin{align}
  \mathcal{E}_{\kappa,\lambda}(m) :=  \int_{\Omega} \left( |\nabla
  \m|^2 -2 \kappa  \m'\cdot  
  \nabla m_3 +\frac{ \lambda}{ | \log \kappa|} |m'|^2 \right) \intd x. 
\end{align}
Notice that the statement of Theorem \ref{prop:existence} remains
valid for minimizers of $\mathcal{E}_{\kappa,\lambda}$.

The following proposition then implies that the $\Gamma$-limit with
respect to the BP-convergence at order $\kappa^2$ is given by
\begin{align}
  \mathcal{E}_0 (R_0,r_0, a_0) := r_0^2( T( a_0) + 8\pi \lambda) -
  2 r_0 \int_{\R^2} (R_0\Phi)' \cdot \nabla \Phi_3 \intd x 
\end{align}
for $(R_0,r_0, a_0 ) \in \widetilde{\mathcal{A}_0}$. 

\begin{proposition}\label{lemma:anisotropy}
  For $\kappa_n \to 0$ as $n \to \infty$, let
  $m_{\kappa_n} \in \mathcal{A}$ BP-converge to
  $(R_0,r_0,a_0) \in \widetilde{\mathcal{A}_0}$.  Then we have
	\begin{align}
          \lim_{n\to \infty} \frac{1}{\kappa_n^2 |\log \kappa_n|}
          \int_{\Omega} |m'_{\kappa_n}|^2 \intd x = 8\pi r_0^2. 
	\end{align}
      \end{proposition}

      In particular, the above result shows that the addition of
        anisotropy does not affect the center of the skyrmion in the
        limit $\kappa \to 0$ in the considered regime. As before,
          the limit of $\mathcal E_0$ is achieved by the N\'eel
          profile, $R_0 = \mathrm{id}$, the center $a_0 =
          \mathrm{argmin}_{a_0 \in \Omega} T(a_0)$ and
          \begin{align}
            r_0 = {4 \pi \over \min_{a_0 \in \Omega} T(a_0) + 8 \pi
            \lambda}. 
          \end{align}
          The corresponding minimal energy is
          \begin{align}
            \min_{\mathcal A_0} \mathcal E_0 = -{16 \pi^2 \over
            \min_{a_0 \in \Omega} T(a_0) + 8 \pi 
            \lambda}. 
          \end{align}

\subsection{A note on free boundary conditions}
\label{sec:free}

We wish to also mention the difference between the behavior of the
energy for BP-converging sequences for the Dirichlet problem
associated with the admissible class $\mathcal A$ and that of the
analogous free problem in which the Dirichlet boundary condition at
$\partial \Omega$ is absent. We point out that in this case the
BP-limit does not give rise to a well-behaved energy whose
minimization would yield the position of the skyrmion in $\Omega$ as
$\kappa \to 0$. Indeed, in the latter case the restriction to $\Omega$
of the N\'eel-type Belavin-Polyakov profile
$\phi_n = R_n \Phi(\rho_n^{-1} (\bullet - a_n))$ with
$R_n = \mathrm{id}$, $\rho_n / \kappa_n \to r_0$ and
$a_n \to a_0 \in \Omega$ in Definition \ref{def:topology} is an
example of a BP-convergent sequence, and, therefore, we have an upper
bound on the Dirichlet energy excess for a sequence of $m_{\kappa_n}$
BP-converging to $(\mathrm{id}, r_0, a_0)$ by
\begin{align}
  Z(\phi_n) = -\int_{\R^2  \setminus \Omega} |\nabla \phi_n|^2 \intd x.
\end{align}

A straightforward computation shows that as $\kappa_n \to 0$ we have
\begin{align}
  \label{eq:H14}
  \begin{split}
    \lim_{n \to \infty} \kappa_n^{-2} Z(\phi_n) & = - \lim_{n \to
      \infty} \kappa_n^{-2} \int_{\R^2 \setminus \Omega} |\nabla
    \phi'_n|^2 \intd
    x \\
    & = -8 r_0^2 \int_{\R^2 \setminus \Omega} {1 \over |x - a_0|^4}
    \intd x = - 4 r_0^2 \int_{\partial \Omega} {(x - a_0) \cdot \nu(x)
      \over |x - a_0|^4} \intd \mathcal H^1(x),
  \end{split}
\end{align}
where $\nu$ is the outward unit normal to $\partial \Omega$, and in
the last line we carried out an integration by parts. This is a
negative contribution that goes to negative infinity as $a_0$
approaches $\partial \Omega$.

For example, if, as in section \ref{sec:halfplane}, we take
$\Omega = \R \times (-\infty, 0)$ and $a_0 = (0, y_0)$ with some
$y_0 \in (-\infty, 0)$, then by \eqref{eq:H14} we have explicitly for
the renormalized energy:
\begin{align}
  \mathcal E_0(\mathrm{id}, r_0, y_0) := \inf_{m_{\kappa_n}} \liminf_{n \to \infty}
  {E_{\kappa_n}(m_{\kappa_n}) - 8 \pi \over \kappa_n^2}   \leq \lim_{n \to \infty}
  {E_{\kappa_n}(\phi_n) - 8 \pi \over \kappa_n^2} 
  = - {2 \pi r_0^2 \over y_0^2} - 8 \pi r_0,
\end{align}
where the infimum is over sequences of $m_{\kappa_n}$ that BP-converge
to $(\mathrm{id}, r_0, a_0)$.  This energy clearly does not have a
minimum in $r_0$, suggesting that the skyrmion is not able to
stabilize its radius at a fixed distance towards the
boundary. Similarly, at fixed radius the skyrmion is attracted towards
the boundary. Dynamically this would give rise to the disappearance of
a skyrmion from $\Omega$ via escape towards the boundary, with zero
energy barrier. This is in contrast with the case of the Dirichlet
boundary conditions considered in section \ref{sec:halfplane}, in
which the exchange contribution has the opposite sign.

Finally, notice that an addition of a sufficiently strong anisotropy
as in section \ref{sec:aniso} may restore existence of local
minimizers. To get some sense for this, consider again a skyrmion in
the half-plane as in the previous paragraph. With the addition of
anisotropy we would then get
\begin{align}
  \mathcal E_0(\mathrm{id}, r_0, y_0) \leq r_0^2 \left( 8 \pi \lambda
  - {2 \pi  \over y_0^2} \right) - 8 \pi r_0, 
\end{align}
and it is clear that the skyrmion should experience a repulsive
interaction and have a well-defined optimal radius far enough from
$\partial \Omega$, while it would still be attracted towards the
boundary close enough to $\partial \Omega$.

\subsection*{Notation and presentation}

Throughout the rest of the paper, we extend $m\in \mathcal{A}$ to
$\R^2$ by $-e_3$.  Furthermore, unless explicitly stated otherwise,
the letters $C, C'$ denote generic, positive constants only depending
on $\Omega$ and which may change from line to line.  Each subsection
first lists its statements and provides a description for their proof
and use throughout the rest of the paper.  The actual proofs are
collected at the end of the respective subsections.

\section{Existence of minimizers}
\label{sec:exist}

We first provide a simple lower bound for the energy that controls the
$L^2$-norm of $\nabla \m$ for sufficiently small $\kappa$.

\begin{lemma}\label{lemma:a_priori}
  For all $\kappa > 0$ and $\m \in H^1(\Omega; \Sphere^2)$ satisfying
  $m = -e_3$ on $\partial \Omega$ we have
  \begin{align}
    \label{eq:Eapriori}
           \mathcal{E}_\kappa(\m) \geq (1 - C \kappa) \int_{\Omega} |\nabla \m |^2
           \intd x,
	\end{align}
        for some $C > 0$ depending only on $\Omega$. Furthermore, if
        also $m\in \mathcal{A}_\kappa$ and $\kappa < 1/(2 C)$ we have
	\begin{align}\label{eq:Dirichlet_excess}
		Z(\m)  \leq 16 \pi C \kappa.
	\end{align}
\end{lemma}

Next, we show by a construction that the infimum energy is strictly
below the topological lower bound for the case of the pure Dirichlet
energy. 

\begin{lemma}\label{lemma:construction}
	For all $\kappa >0 $ we have
	\begin{align}
          \label{eq:Eless8pi}
		\inf_{\mathcal{A}} \mathcal{E}_\kappa < 8\pi.
	\end{align}
	In particular, the restricted admissible sets
        $\mathcal{A}_\kappa$, see definition
        \eqref{def:admissible_restricted}, are non-empty.
        Furthermore, there exist constants $C>0$ and $\kappa_0>0$
        depending only on $\Omega$ such that for all
        $\kappa\in (0,\kappa_0)$, we have
	\begin{align}
          \inf_{\mathcal{A}} \mathcal{E}_\kappa \leq  8\pi  - C \kappa ^2.
	\end{align}        
\end{lemma}

\begin{proof}[Proof of Lemma \ref{lemma:a_priori}]
  For $\m \in H^1(\Omega; \Sphere^2)$ satisfying $m = -e_3$ on
  $\partial \Omega$ we have by Cauchy-Schwarz and Poincar\'e
  inequalities
	\begin{align}\label{eq:CS_lower_bound}
		\begin{split}
                  \mathcal{E}_\kappa(\m) & \geq \int_{\Omega} |\nabla
                  \m |^2 \intd x - 2 \kappa \left( \int_\Omega |m'|^2
                    \intd x \int_{\Omega} |\nabla \m_3|^2 \intd x
                  \right)^\frac{1}{2}\\
                  & \geq \int_{\Omega} |\nabla \m |^2 \intd x - C
                  \kappa \left( \int_\Omega |\nabla m'|^2 \intd x
                    \int_{\Omega} |\nabla \m_3|^2 \intd x
                  \right)^\frac{1}{2},
		\end{split}
	\end{align}
        from which \eqref{eq:Eapriori} follows. Furthermore, under the
        assumption $m\in \mathcal{A}_\kappa$ we have
        $\mathcal{E}_\kappa(\m) < 8\pi$, so combining this with
        \eqref{eq:Eapriori} and a bound on $\kappa$ we obtain
        $\int_\Omega |\nabla m|^2 \intd x <16 \pi$. Using this fact
        together with \eqref{eq:Eapriori}, we obtain
        \eqref{eq:Dirichlet_excess}.
        \end{proof}

\begin{proof}[Proof of Lemma \ref{lemma:construction}]
	\textit{Step 1: Truncation of the Belavin-Polyakov profile}
	
	We truncate the standard Belavin-Polyakov profile
	by choosing $L>1$ and setting
	\begin{align}
		f_L(r) := \begin{cases}
			\frac{2r}{1+r^2} & \text{ if } r<L,\\
			\frac{2}{1+L^2}(2L-r) & \text{ if } L \leq r <2L,\\
			0 & \text{ if } 2L \leq r,
		\end{cases}
	\end{align}
	for $r >0$ and
	\begin{align}
          \bp_L(x) := \begin{pmatrix}
            - f_L(|x|) \frac{x}{|x|}, \operatorname{sign}(1-|x|)
            (1-f_L^2(|x|))^\frac{1}{2} 
          \end{pmatrix}
	\end{align}
	for $x\in \R^2$.

	One may then compute, see \cite[equation (A.66)]{bernand2019quantitative}, that
	\begin{align}
		|\nabla \bp_L|^2(x) = \frac{|f'_L|^2(|x|)}{1-f_L^2(|x|)} + \frac{f_L^2(|x|)}{|x|^2},
	\end{align}
	so that
	\begin{align}
		\int_{B_L(0)} |\nabla \bp_L|^2(x) \intd x = \frac{8\pi L^2}{1+L^2}
	\end{align}
	and
	\begin{align}
			|\nabla \bp_L|^2(x) \leq C L^{-4}
	\end{align}
        for all $x \in B_{2L}(0) \setminus B_L(0)$.  Consequently, we
        have
	 \begin{align}
           \int_{B_{2L}(0)} |\nabla \bp_L|^2 \intd x \leq  8\pi + C L^{-2}.
	 \end{align}
         Furthermore, as can be seen by a direct computation we have
 	\begin{align}
          -2  \int_{B_{2L}(0)}  \bp_L' \cdot \nabla \bp_{L,3}
          \intd x \leq  - 8\pi  + C L^{-2}.
 	\end{align}
 
 	\textit{Step 2: Construction of competitors}
 	
	Let $r>0$ be the in-radius.  Without loss of generality, we
        may assume that the corresponding in-circle is centered at the
        origin, so that $B_r(0) \subset \Omega$.  Let now $\rho >0$
        and $L>1$ be such that $2L \rho \leq r$.  Then the function
        $\phi_{\rho,L}(x):= \Phi_L(\rho^{-1}x)$ satisfies
        $\phi_{\rho,L} \in \mathcal{A}$.  For $\kappa < 1$ we compute
 	\begin{align}
          \mathcal{E}_\kappa(\mathbf{\phi}_{\rho,L}) \leq  8\pi - 8
          \pi \kappa \rho + CL^{-2}.
 	\end{align}
 	To minimize the preceding expression, we need to choose $\rho$
        as big as possible, i.e., $\rho = \frac{r}{2L}$. This yields
 	\begin{align}
          \mathcal{E}_\kappa(\mathbf{\phi}_{\rho,L}) \leq  8\pi - 4\pi
          \kappa r L^{-1} + 
          CL^{-2}. 
 	\end{align}
 	In particular, choosing $L$ big enough, we obtain that
 	\begin{align}
 		\mathcal{E}_\kappa(\mathbf{\phi}_{\rho,L}) < 8\pi,
 	\end{align}
 	which yields \eqref{eq:Eless8pi}. Furthermore, optimizing in
        $L$ gives $L = c / ( r \kappa )$ for some suitably chosen
        $c > 0$ depending only on $\Omega$, so that $L > 1$ for
        $\kappa < c / r$ and
 	\begin{align}
 		\mathcal{E}_\kappa(\mathbf{\phi}_{\rho,L}) \leq 8\pi -
          C r^2 \kappa^2, 
 	\end{align}
        which completes the proof.
\end{proof}

\begin{proof}[Proof of Theorem \ref{prop:existence}]
  Let $(\m_n) \in \mathcal{A}$ be a minimizing sequence.  By Lemma
  \ref{lemma:a_priori}, assuming that $\kappa$ is small enough, we get
  that $(\m_n)$ is uniformly bounded in $H^1(\Omega;\mathbb{S}^2)$.
  Consequently, there exists a subsequence (not relabeled) and
  $\m_\infty \in H^1(\Omega; \mathbb{S}^2)$ such that
  $\m_n \to \m_\infty $ in $L^2$ and
  $\nabla \m_n \rightharpoonup \nabla \m_\infty$ in $L^2$ as
  $n \to \infty$.  Furthermore, by a weak-times-strong argument, we
  get
  $\int_{\R^2} \m_n' \cdot \nabla m_{n,3} \intd x \to \int_{\R^2}
  \m_\infty' \cdot \nabla m_{\infty,3} \intd x$ and, therefore, we
  have
  \begin{align}    
    \mathcal{E}_\kappa(\m_\infty) \leq \liminf_{n\to \infty} 
    \mathcal{E}_\kappa (\m_n) = \inf_{\mathcal A} \mathcal E_\kappa. 
  \end{align}
  Thus, it remains to prove that $m_\infty \in \mathcal A$, i.e., that
  $\mathcal N(m_\infty) = 1$.
        
  Arguing as in \cite{brezis1983large,melcher14}, we complete the
  squares to get for all $\m \in H^1(\Omega; \mathbb{S}^2)$ that
  \begin{align}
    \label{eq:squares}
    \int_\Omega |\nabla \m|^2 \intd x  \pm 8\pi \mathcal N (\m) 
    &= \int_\Omega |\partial_1 \m \mp \m \times \partial_2 \m|^2 \intd x.
  \end{align}
  As a result, by the lower semicontinuity of the right-hand side in
  \eqref{eq:squares} and the continuity of the DMI term we have
  \begin{align}
    \mathcal{E}_\kappa(\m_\infty) \pm 8\pi
    \mathcal{N}(\m_\infty) \leq \liminf_{n \to \infty} \left( 
    \mathcal{E}_\kappa (\m_n) \pm 8\pi \mathcal{N}(\m_n) \right) = 
    \liminf_{n\to \infty} \mathcal{E}_\kappa (\m_n) \pm 8\pi. 
  \end{align}
  Therefore, for small enough $\kappa$ we get with the help of Lemmas
  \ref{lemma:a_priori} and \ref{lemma:construction}:
  \begin{align}
    \pm 8\pi \mathcal{N}(\m_\infty) \leq
    \mathcal{E}_\kappa(\m_\infty) \pm 8\pi 
    \mathcal{N}(\m_\infty) \leq  \liminf_{n\to \infty}
    \mathcal{E}_\kappa (\m_n) 
    \pm 8\pi < 8 \pi  \pm 8\pi, 
  \end{align}
  from which $\mathcal{N}(\m_\infty) = 1$ immediately follows.
\end{proof}

\section{The next-order $\Gamma$-limit}
\label{sec:Gamma}

\subsection{Preliminaries}

Before we turn to the actual proof of Theorem \ref{thm:convergence},
we establish a number of preliminary statements designed to provide
compactness in BP-convergence.  First, we prove that in fact the
skyrmion center $ a_\kappa$ for $0 < \kappa \ll 1$ satisfies
${ a_\kappa } \in \Omega$, as well as a lower bound for the Dirichlet
excess of a minimizer in terms of its radius $\rho$ and
$\operatorname{dist}({ a_\kappa}, \R^2 \setminus \Omega)$.  The idea
is that for $\phi$ achieving the Dirichlet distance from $m$ to
$\mathcal{B}$ we have $\nabla( m - \phi)(x) = - \nabla \phi(x)$ for
all $x\in \R^2\setminus \Omega$, so that the control of Theorem
\ref{thm:rigidity} can be translated into a control over the radius
and the center.

\begin{lemma}\label{lemma:lower_bound_dirichlet}
  There exist $\kappa_0>0$ and $C, C' >0$ depending only on $\Omega$
  such that for all $0 < \kappa <\kappa_0$ the following statement
  holds:

  Let $m\in \mathcal{A}_\kappa$ and let
  $\phi(x)= R \Phi (\rho^{-1}(x- a))$ with $R \in SO(3)$, $\rho >0$,
  $a \in \R^2$ achieve the Dirichlet distance from $\mathcal B$ to
  $m$.  Then we have $a \in \Omega$ and
	\begin{align}\label{eq:radius_vs_distance}
          \frac{\rho^2}{\operatorname{dist}^2 (a, \R^2
          \setminus \Omega) } \leq C Z(\m) \leq C' \kappa. 
	\end{align}	
\end{lemma}

We now record some basic estimates for the skyrmion tail. As we wish
to apply this result also in the construction of a recovery sequence,
we take care to only assume $m \in \mathcal{A}$, not
$m \in \mathcal{A}_\kappa$.

\begin{lemma}\label{lemma:skyrmion_tail}
  There exists a constant $C>0$ only depending on $\Omega$ such that
  we have the following statement: 
  
  Let $\phi = R \Phi (\rho^{-1}(\bullet-a))$ with $R \in SO(3)$,
  $\rho \in \left( 0,\frac{1}{2} \right)$, ${ a} \in \Omega$.  Then,
  for all $x\in \R^2\setminus \Omega$ we have
	\begin{align}
          \left| \frac{1}{\rho} (\phi(x) - Re_3) - 2R \left( \frac{x-{
          a}}{\abs{x-{ a}}^2} , 0\right) \right|
          & \leq C
            \frac{\rho}{|x-{
            a}|^2}, 
            \label{eq:convergence_tail_exterior}\\
          \left| \frac{1}{\rho}  \nabla \phi (x) - 2\nabla \left( R
          \left( \frac{x-{ a}}{\abs{x-{ a}}^2} , 0\right)
          \right)\right|
          & \leq  \frac{\rho}{|x-{ a}|^3} 
		\label{eq:convergence_tail_gradient_exterior}
	\end{align}
	Let furthermore $m\in \mathcal{A}$. For $w:= m +e_3 - \phi
        -Re_3$ we have
	\begin{align}
          \int_{\Omega} |w|^2 \intd x
          & \leq C \left( \rho^2 + \int_{\R^2} |\nabla (m-\phi) |^2
            \intd x \right). \label{eq:L2_estimate_tail_upper} 
	\end{align}
\end{lemma}

With these bounds we can give an estimate for the DMI term.  Again, we
only assume $m\in \mathcal{A}$ to keep the statement applicable for
the construction of the recovery sequence.

\begin{lemma}\label{lemma:lower_bound_DMI}
  Let $\kappa \in (0,1)$, $m \in \mathcal{A}$ and
  $\phi = R \Phi(\rho^{-1}(\bullet - a))$ with $R \in SO(3)$,
  $\rho \in \left( 0,\frac{1}{2} \right)$, $a \in \Omega$.  Then there
  exist $C_1 >0$ universal and $C_2 = C_2\left(\Omega, a \right)>0$
  such that the following holds:
	\begin{align}\label{eq:DMI_bounded}
          \left| \int_{\R^2} (R (\Phi+e_3))'\cdot \nabla (R
          \Phi )_3  \intd x \right| \leq C_1.
	\end{align}
	and
  \begin{multline}
    \label{eq:lower_bound_DMI_for_recovery}
    \qquad \left|-2\kappa \int_\Omega \m' \cdot \nabla m_{3} \intd x +
      2 \kappa \rho \int_{\R^2} (R (\Phi +e_3))' \cdot \nabla (R
      \Phi )_3 \intd x \right| \\
    \leq C_2 \left( \int_{\R^2}|
      \nabla (m- \phi)|^2\intd x + \rho^\frac{6}{5} \right)\kappa
    \qquad
	\end{multline}
	If additionally $m\in \mathcal{A}_\kappa$ and $\phi$ achieves
        the Dirichlet distance of $m$ to $\mathcal{B}$, then there
        exists $\kappa_0>0$ and $C_3 > 0$ depending only on $\Omega$
        such that for all $\kappa \in (0, \kappa_0)$ the estimate
        takes the form
	\begin{align}\label{eq:lower_bound_DMI}
          \left|-2\kappa \int_\Omega \m' \cdot \nabla
          m_{3} \intd x + 2 \kappa \rho \int_{\R^2} (R
          (\Phi +e_3))' \cdot \nabla (R
          \Phi )_3 \intd x \right| 
          \leq C_3 \left( Z (m) +
          \rho^\frac{6}{5} \right)\kappa.
	\end{align}
\end{lemma}

\begin{proof}[Proof of Lemma \ref{lemma:lower_bound_dirichlet}]
  \textit{Step 1 : ${ a} \in \operatorname{conv}(\Omega)$, the convex
    envelope of $\Omega$.}
	
  Towards a contradiction, we assume that
  ${ a} \not \in \operatorname{conv}(\Omega)$.  Then there exists
  $n \in \Sphere^1$ such that $ { a} \cdot n \geq x \cdot n$ for all
  $x \in \Omega \subset \operatorname{conv}(\Omega)$.  Consequently,
  $\{ x \in \R^2 : (x-a)\cdot n > 0\} \subset {\R^2\setminus \Omega}$
  and by the estimates \eqref{eq:rigidity} and
  \eqref{eq:Dirichlet_excess}, recalling that $m(x)=-e_3$ for all
  $x\in{\R^2\setminus \Omega}$, we have
	\begin{align}
          4\pi = \int_{\{ (x- { a}) \cdot n > 0\} } |\nabla \phi |^2
          \intd x \leq \int_{\R^2\setminus\Omega} |\nabla \phi |^2
          \intd x \leq \int_{\R^2} |\nabla (\phi - \m) |^2 \intd x
          \leq C \kappa. 
	\end{align}
        For small enough $\kappa$ we have a contradiction.
		
	\textit{Step 2: There exist $C,C' >0$ such that
		\begin{align}\label{eq:rhok}
			\rho^2 \leq C Z(\m) \leq  C'  \kappa.
		\end{align}
              }We have
              $\operatorname{diam}(\Omega) =
              \operatorname{diam}(\operatorname{conv}(\Omega))$.
              Since from Step 1 we know that
              $a \in \operatorname{conv}(\Omega)$, it follows that
              $\Omega \subset B_{\operatorname{diam}(\Omega)} ( { a
              })$. By a direct computation (as in \cite[Equation
              (A.67)]{bernand2019quantitative}) and
              \eqref{eq:rigidity}, we have
	\begin{align}
	  \begin{split}
            \frac{8\pi}{1+\rho^{-2}\operatorname{diam}^2(\Omega) } &=
            \int_{\R^2\setminus
              B_{\rho^{-1}\operatorname{diam}(\Omega)} ( 0)} |\nabla
            \bp |^2 \intd x = \int_{\R^2\setminus
              B_{\operatorname{diam}(\Omega)} (
              a)} |\nabla \phi |^2 \intd x \\
            & \leq \int_{\R^2\setminus \Omega} |\nabla \phi |^2 \intd
            x \leq \int_{\R^2} |\nabla (\phi-m) |^2 \intd x \leq C
            Z(m).
	  \end{split}
	\end{align} 
	Therefore, together with \eqref{eq:Dirichlet_excess} and
        taking $\kappa$ small enough, we have
        $\rho^{2} \leq C Z(\m)\leq C' \kappa$.
	
	\textit{Step 3: ${ a} \in \Omega$, provided $\kappa$ is
            small enough.}
	
          Towards a contradiction, let us assume that
          ${ a } \not \in \Omega$. We claim that since $\Omega$ is a
          Lipschitz domain, there exist $\alpha>0$ and $\tilde r>0$
          depending only on $\Omega$ such that
          $\mathcal{C}_\alpha(a )\cap B_{\tilde r}(a ) \subset \R^2
          \setminus \Omega$, where $\mathcal{C}_\alpha(a )$ is a cone
          with vertex at $a$ and the opening angle $\alpha$.  Indeed,
          we may assume that $a$ is sufficiently close to $\Omega$,
          and near $a$ the set $\Omega$ is locally a subgraph of a
          Lipschitz function. Translating the point $a$ vertically
          down towards $\partial \Omega$, we obtain a cone
          $\mathcal C_\alpha(\tilde a)$ pointing up with
          $\tilde a \in \partial \Omega$ that lies above $\Omega$ in
          $B_r(\tilde a)$ for some $\alpha > 0$ and $r > 0$ depending
          only on $\Omega$. Hence the claim follows by translating the
          cone $\mathcal C_\alpha(\tilde a)$ vertically upward until
          its vertex coincides with $a$.

          We now compute (again, as in \cite[Equation
          (A.67)]{bernand2019quantitative}, and using
          \eqref{eq:rigidity} and \eqref{eq:Dirichlet_excess})
	\begin{align}\label{eq:aom}
		\begin{split}
                  \frac{\rho^{-2}\tilde r^2}{1+\rho^{-2}\tilde r^2 }
                  \leq C \int_{(\mathcal{C}_\alpha(a) - a) \cap
                    B_{\rho_\kappa^{-1}\tilde r}(0)} |\nabla \bp |^2
                  \intd x &= C \int_{\mathcal{C}_\alpha(a) \cap
                    B_{\tilde r}(a)} |\nabla
                  \phi|^2 \intd x \\
                  &\leq C \int_{\R^2\setminus \Omega} |\nabla \phi |^2
                  \intd x \leq C \kappa .
		\end{split}
	\end{align}
	By step 2, the left-hand side is uniformly bounded from
        below, giving a contradiction for $\kappa $ small enough.
	
	\textit{Step 4: We have estimate
          \eqref{eq:radius_vs_distance}.}
	
        Once again, since $\Omega$ is a Lipschitz domain there exist
        $\alpha>0$ and $\tilde r>0$ depending only on $\Omega$ such
        that for any $\hat x \in \partial \Omega$ we have
        $\mathcal{C}_\alpha(\hat x) \cap B_{\tilde r}(\hat x) \subset
        \R^2 \setminus \Omega$. As $a \in \Omega$, there is
        $\hat x \in \partial \Omega$ such that
        $|\hat x - { a } | = \operatorname{dist}({ a },\R^2\setminus
        \Omega)$. We next fix $\delta>0$ small enough (depending only
        on $\alpha$). If
        $ \operatorname{dist}({ a },\R^2\setminus \Omega) \geq \delta
        \tilde r$ then by the estimate \eqref{eq:rhok} shown in Step 2
        we have
        $\frac{\rho^2}{\operatorname{dist}^2 (a, \R^2 \setminus
          \Omega) } \leq \frac{C \kappa}{\delta^2 \tilde r^2} \leq C'
        \kappa $ with $C'$ depending only on $\Omega$.  If, on the
        other hand,
        $\operatorname{dist}({ a },\R^2\setminus \Omega) < \delta
        \tilde r$ (meaning $|\hat x-a|$ is very small comparing to
        $\tilde r$) then using basic geometry arguments we deduce that
        there exist a cone $ \mathcal{\hat C}_ \beta(a)$ with the
        opening angle $\beta > 0$, and a constant $c>0$ (both
        depending only on $\alpha$ and $\delta$) such that
        $(\mathcal{\hat C}_\beta(a) \cap B_{\frac{\tilde r}{2}}(a))
        \setminus B_{c|\hat x - { a }|}(a) \subset
        \mathcal{C}_\alpha(\hat x) \cap B_{\tilde r}(\hat x) \subset
        \R^2 \setminus \Omega$, see Figure \ref{fig:cones}.
        
        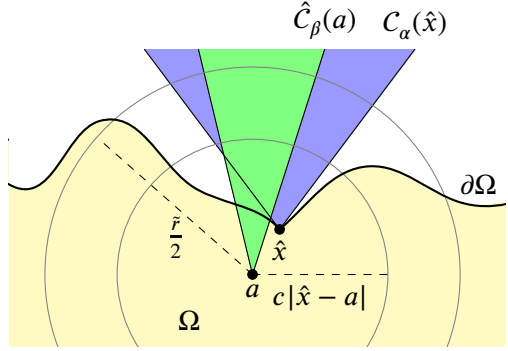
\begin{figure}
        	\centering
        		\begin{tikzpicture}[scale=1.2]
                          \clip (-3,2.7) -- (2.5,2.7) -- (2.5,-1.3) --
                          (-3,-1.3); \fill[yellow!30] (0 ,0) to [ out
                          angle=130,curve through={(-1,.5) .. (-2,1.2)
                            .. (-3,.5)}] (-3,1) -- (-3,-3) --
                          (0,-3)--(0,0); \fill[yellow!30] (0 ,0) to [
                          out angle=40,curve through={(1,.7) .. (2,.3)
                            .. (3,.5)}] (4,1) -- (4,-4) -- (0,-4) --
                          (0,0); \fill[blue!40] (-1.5,2) -- (0,0) --
                          (1.5,2) -- cycle; \fill[green!50] (-.9,2) --
                          (-.3,-.5) -- (.5,2) -- cycle; \filldraw
                          (-1.5,2) -- (0,0) circle (1.5pt) node[below]
                          {$\hat x$} -- (1.5,2) node[above]
                          {$\mathcal{C}_\alpha(\hat x)$}; \filldraw
                          (-.9,2) -- (-.3,-.5) circle (1.5pt)
                          node[below] {$a$} -- (.5,2) node[above]
                          {$\mathcal{\hat C}_\beta(a)$}; \draw[gray]
                          (-.3,-.5) circle (1.5); \draw[gray]
                          (-.3,-.5) circle (2.3); \draw[dashed]
                          (-.3,-.5) -- node[below] {$c |\hat x - a|$}
                          (1.2,-.5); \draw[dashed] (-.3,-.5) --
                          node[below] {$\frac{\tilde r}{2}$} (-2,1);
                          \draw[thick] (0 ,0) to [ out angle=130,curve
                          through={(-1,.5) .. (-2,1.2) .. (-3,.5)}]
                          (-3,1); \draw[thick] (0 ,0) to [ out
                          angle=40,curve through={(1,.7) .. (2,.3)
                            .. (3,.5)}] (-4,1); \node at (2.2,.5)
                          {$\partial \Omega$}; \node at (-1,-1)
                          {$\Omega$};
		\end{tikzpicture}
                \caption{\label{fig:cones} Sketch of $\Omega$ and the
                  cones $\mathcal{C}_\alpha(\hat x)$ and
                  $\mathcal{\hat C}_\beta(a)$. }
        \end{figure}

        Similarly to the previous
        calculations (see \eqref{eq:aom}), we obtain
	\begin{align}
          \frac{\rho^{-2}\frac{\tilde
          r^2}{4}}{1+\rho^{-2}\frac{\tilde r^2}{4}}
          -\frac{\rho^{-2} c^2|\hat x - { a }|^2}{1+\rho^{-2}
          c^2 |\hat x - {
          a}|^2 }   \leq C \int_{ ( \hat{\mathcal{C}}_\beta(a) \cap B_{\tilde
          r \over 2}(a) ) \setminus B_{c |\hat x - a|}(a)}
          |\nabla \phi|^2 \intd x  \leq CZ(\m).  
	\end{align}
	We further calculate
	\begin{align}
	  \begin{split}
            \frac{\rho^{-2}\frac{\tilde
                r^2}{4}}{1+\rho^{-2}\frac{\tilde r^2}{4}}
            -\frac{\rho^{-2}c^2|\hat x - { a }|^2}{1+\rho^{-2}
              c^2|\hat x - { a }|^2 } & = 1 -
            \frac{1}{1+\rho^{-2}\frac{\tilde r^2}{4}} -
            \frac{\rho^{-2} c^2|\hat x - { a }|^2}{1+\rho^{-2}
              c^2|\hat x -
              { a }|^2}\\
            & = \frac{1}{1+\rho^{-2} c^2|\hat x - { a }|^2} -
            \frac{1}{1+\rho^{-2}\frac{ \tilde r^2}{4}}\\
            & \geq \frac{1}{1+\rho^{-2} c^2|\hat x - { a}|^2} -\frac{4
              \rho^2}{\tilde r^2}.
	  \end{split}
	\end{align}
	By step 2, we know that $\rho^2 \leq CZ(\m)$ and, therefore,
        we have
	\begin{align}
          \frac{1}{1+\rho^{-2} c^2|\hat x - { a }|^2}  \leq  C Z(m).
	\end{align}
	Taking $\kappa$ small enough and recalling $Z(\m) \leq C
        \kappa$, we deduce that
	\begin{align}
          \frac{\rho^{2}}{\operatorname{dist}^2({ a },
          \R^2 \setminus \Omega)} \leq CZ(\m) \leq  C' \kappa,
	\end{align}
        as claimed.
\end{proof}

\begin{proof}[Proof of Lemma \ref{lemma:skyrmion_tail}]
  If $x\in\R^2\setminus \Omega$, then
  $w(x) = -R \left( \Phi(\rho^{-1} (x - a)) + e_3\right)$. It is
  straightforward to compute
  \begin{align}
    \frac{1}{\rho} \left( \Phi(\rho^{-1} (x - a)) + e_3\right)  =
    \left( -\frac{2(x-a)}{\rho^2+|x-a|^2}, \frac{2\rho}{\rho^2+|x-a|^2}
    \right). 
  \end{align}
Therefore, we have  
\begin{align}\label{eq:ident}
  \frac{1}{\rho} w(x)- 2R\left(\frac{x-a}{|x-a|^2},0 \right)= -R \left(
  \frac{2\rho^2 (x-a)}{(\rho^2+|x-a|^2)|x-a|^2},
  \frac{2\rho}{\rho^2+|x-a|^2} \right). 
\end{align}
Using the fact that $\rho^2+|x-a|^2\geq 2\rho |x-a|$ and the fact that
$R \in SO(3)$, we obtain \eqref{eq:convergence_tail_exterior}. Taking
the gradient of both parts of \eqref{eq:ident} we arrive at
\eqref{eq:convergence_tail_gradient_exterior} in a similar way.

For $x \in \R^2 \setminus B_{\operatorname{diam}(\Omega)}(a)$,
  from estimate \eqref{eq:convergence_tail_exterior} we get
	\begin{align}\label{eq:estimate_w}
		| w(x)| \leq  {C \rho},
	\end{align}
        Therefore, with the help of Friedrichs' inequality
        \cite[Corollary 6.11.2]{mazya} we obtain
	\begin{align}
	  \begin{split}
            \int_{\Omega} |w|^2 \intd x & \leq
            \int_{B_{\operatorname{diam}(\Omega)}(a)} |w|^2 \intd x
            \leq C \left( \int_{B_{\operatorname{diam}(\Omega)}(a)}
              |\nabla w|^2 \intd x +
              \int_{\partial B_{\operatorname{diam}(\Omega)}(a)} |w|^2
              \intd \mathcal H^1(x) \right) \\ 
            & \leq C \left( \rho^2 + \int_{\R^2} |\nabla (m-\phi) |^2
              \intd x \right).
	  \end{split}
	\end{align}
	which is the estimate \eqref{eq:L2_estimate_tail_upper}.
\end{proof}

\begin{proof}[Proof of Lemma \ref{lemma:lower_bound_DMI}]
	Letting $w:= m + e_3 - \phi - Re_3$, we compute
	\begin{align}\label{eq:split_DMI}
	  \begin{split}
            \int_\Omega m'\cdot \nabla m_{3} \intd x
            &= \int_\Omega (m+e_3)'\cdot \nabla (m +e_3)_3 \intd x\\
            &= \int_\Omega (\phi +Re_3 )'\cdot\nabla
            (\phi_{3} +Re_3 ) \intd x
            + \int_\Omega (\phi +Re_3 )'\cdot\nabla w_{3} \intd x\\
            & \quad + \int_\Omega w'\cdot\nabla (\phi +Re_3 )
            \intd x + \int_\Omega w'\cdot\nabla w_{3} \intd x.
	  \end{split}
	\end{align}
	
	The first term gives
	\begin{align}
          \int_\Omega (\phi +Re_3 )'\cdot\nabla (\phi_{3} +Re_3
          ) \intd x = \rho \int_{\rho^{-1}(\Omega -{ a })} (R
          (\Phi+e_3 ))'\cdot\nabla (R\Phi)_3 \intd x.  
	\end{align}
	As $|(R (\Phi+e_3 ))'\cdot\nabla (R\Phi)_3| \leq C / (1 +
        |x|^3)$ uniformly in $R$, we get
	\begin{align}
	  \begin{split}
            \left|\rho \int_{\R^2\setminus
                B_{\rho^{-1}\operatorname{dist}({ a
                  },\R^2\setminus\Omega)}(0)} (R (\Phi+e_3
              ))'\cdot\nabla (R\Phi)_3 \intd x \right| & \leq C\rho
            \int_{\rho^{-1}\operatorname{dist}({ a
              },\R^2\setminus\Omega)}^\infty \frac{r}{1 + r^3}
            \intd r
            \\
            & \leq C \frac{\rho^2}{ \operatorname{dist}({ a
              },\R^2\setminus\Omega)}.
	  \end{split}
	\end{align}
	Similarly, we get the estimate \eqref{eq:DMI_bounded}.  In
        total, we obtain
	\begin{align}\label{eq:DMI_first_term}
          \left| \int_\Omega (\phi +Re_3 )'\cdot\nabla (\phi +Re_3 )
          \intd x  -  \rho \int_{\R^2} (R (\Phi+e_3 ))'\cdot\nabla
          (R\Phi)_3 \intd x \right|  \leq  C\frac{\rho^2}{
          \operatorname{dist}({ a},\R^2\setminus\Omega)}. 
	\end{align}
	
	We treat the second term by using Young's inequality to get
	\begin{align}
          \left| \int_\Omega (\phi +Re_3 )'\cdot\nabla w_{3} \intd
          x\right| \leq  \frac{1}{2}\left(
          \rho^2\int_{\rho^{-1}(\Omega-{ a})} |\Phi + e_3|^2 \intd x +
          \int_{\Omega} |\nabla w|^2 \intd x \right). 
	\end{align}
	As $|\Phi +e_3| \leq C / ( 1 + |x|)$, we have
	\begin{align}\label{eq:L2_of_BP}
          \int_{\rho^{-1}(\Omega-{ a })} |\Phi + e_3|^2
          \intd x \leq  C \int_0^{\rho^{-1}\operatorname{diam}(\Omega)}
          \frac{r}{(1+r)^2} \intd r \leq C' |\log \rho|.
	\end{align}
	Therefore, we get
	\begin{align}\label{eq:DMI_second_term}
          \left| \int_\Omega (\phi +Re_3 )'\cdot\nabla w_{3} \intd
          x \right| \leq  C \left( \rho^2|\log\rho| + \int_{\R^2}|
          \nabla (m- \phi )|^2\intd x \right). 
	\end{align}
	
	For the third term, we find by H\"older's inequality that for
        $p>2$ and $p' = \frac{p}{ p - 1} \in (1,2)$ we have
	\begin{align}\label{eq:Hoelder}
	  \begin{split}
            \left|\int_\Omega w'\cdot\nabla (\phi +Re_3 ) \intd
              x\right| &\leq \| w' \|_{L^p(\Omega)} \| \nabla
            \phi\|_{L^{p'}(\Omega)}.
	  \end{split}
	\end{align}
	Noticing that $|\nabla \Phi|^{p'}$ decays sufficiently fast to
        be integrable, we furthermore compute
	\begin{align}
          \int_\Omega |\nabla \phi|^{p'} \intd x \leq \rho^{{2-p'}}
          \int_{\R^2} |\nabla \Phi |^{p'} \intd x \leq C
          \rho^{{2-p'}}. 
	\end{align}
	By the Sobolev embedding, for some $C_p > 0$ depending only on
        $\Omega$ and $p$ we have
	\begin{align}
	  \begin{split}
            \|w'\|_{L^p(\Omega)} & \leq C_p \left( \|\nabla
              w'\|_{L^2(\Omega)}+ \|w'\|_{L^2(\Omega)} \right).
	  \end{split}
	\end{align}
	Therefore, together with the estimates
        \eqref{eq:L2_estimate_tail_upper} and \eqref{eq:Hoelder} we see that
	\begin{align}
                  \left| \int_\Omega w'\cdot\nabla (\phi +Re_3 )
                    \intd x\right| &\leq C_p \rho^{\frac{2}{p'} -1}\left(\rho
                    + \left(\int_{\R^2}| \nabla (m- \phi)|^2\intd
                      x\right)^\frac{1}{2}
                  \right).
      \end{align}
      Applying Young's inequality to
      $ \rho^{\frac{2}{p'} -1} \left(\int_{\R^2}| \nabla (m-
        \phi)|^2\intd x\right)^\frac{1}{2}$ and choosing $p=5$, we obtain
  	\begin{align}\label{eq:DMI_third_term}
  		\begin{split}
                  \left| \int_\Omega w'\cdot\nabla (\phi +Re_3 ) \intd
                    x\right| & \leq C_p \left( \rho^\frac{2}{p'} +
                    \rho^{\frac{4}{p'} -2} +\int_{\R^2}| \nabla (m-
                    \phi)|^2\intd x
                  \right) \\
                  & \leq C\left( \rho^{\frac{6}{5}} +\int_{\R^2}|
                    \nabla (m- \phi)|^2\intd x \right) .
        \end{split}
	\end{align}

	For the last term, we have by Young's inequality and
        estimate \eqref{eq:L2_estimate_tail_upper} that
	\begin{align}\label{eq:DMI_fourth_term}
          \left| \int_\Omega w'\cdot\nabla w_{3} \intd x \right| \leq
          C\left( \rho^2 + \int_{\R^2} |\nabla (m-\phi)|^2 \intd x
          \right). 
	\end{align}
        Combining the estimates \eqref{eq:DMI_first_term},
        \eqref{eq:DMI_second_term}, \eqref{eq:DMI_third_term}, and
        \eqref{eq:DMI_fourth_term} in \eqref{eq:split_DMI}, we get the
        desired estimate \eqref{eq:lower_bound_DMI_for_recovery}.
        Furthermore, the only dependence of the constant $C$ on $a$ in
        \eqref{eq:lower_bound_DMI_for_recovery} is through estimate
        \eqref{eq:DMI_first_term}, and for $\phi$ achieving the
        Dirichlet distance we can uniformly absorb this term into
        $Z(m)$ using Lemma \ref{lemma:lower_bound_dirichlet}. This
        gives us estimate \eqref{eq:lower_bound_DMI}.
\end{proof}

\subsection{$\Gamma$-convergence}

With the preliminary statements above, we can now argue for all the
relevant compactness properties.  Essentially, the centers $ a_\kappa$
cannot approach the boundary since otherwise the Dirichlet excess will
be too large by Lemma \ref{lemma:lower_bound_dirichlet}.  Estimates
for the radii $\rho_\kappa$ and the Dirichlet excess $Z(m_\kappa)$
easily follow.  To control pinning of the rotation, we refer to the
Moser-Trudinger-type inequality \cite[Lemma
2.5]{bernand2019quantitative} to side-step the fact that in two
dimensions $H^1$ does not embed into $L^\infty$.  Recall that
$\mathcal{A}_0$ and $\widetilde{\mathcal{A}_0}$ are defined in
\eqref{def:A0} and \eqref{def:A0tilde}, respectively.

\begin{lemma}\label{lemma:compactness}
	
  For every sequence of $\kappa_n \to 0$ and
  $m_{\kappa_n} \in \mathcal{A}_{\kappa_n}$ with
  \begin{align} \label{eq:negative_energy} \limsup_{n\to \infty }
    \frac{\mathcal{E}_{\kappa_n}(m_{\kappa_n}) - 8\pi}{\kappa_n^2}<0
	\end{align}
	there exists a subsequence (not relabeled) and
        $(R_0,r_0, a_0 ) \in \widetilde{\mathcal{A}_0}$ such that
        $m_{\kappa_n}$ BP-converges to $(R_0,r_0, a_0)$.
      \end{lemma}
      \noindent Note that the above lemma does not yet yield a limit
      in $\mathcal A_0$.

      We now proceed to formulate the $\Gamma$-convergence result by
      first giving the lower bound statement.  The convergence of the
      DMI term will follow from Lemma \ref{lemma:lower_bound_DMI}, and
      therefore we only need to deal with the Dirichlet excess. In
      particular, we have to prove that deviations from the
      Belavin-Polyakov profile are only energetically favorable in the
      tail of the skyrmion.  To this end, we split the Dirichlet
      excess into a part localized in the core and a tail
      contribution.  The localized part turns out to be given by the
      Hessian of the Dirichlet energy after using a new
      parametrization by mapping $\R^2$ to the sphere, with the
      Belavin-Polyakov profile closest to $m$.  Since Belavin-Polyakov
      profiles are minimizers of the Dirichlet energy, the Hessian is
      non-negative and thus the contribution of any possible core
      correction is non-negative.

\begin{proposition}\label{prop:lower_bound}
  Let $\kappa_n \to 0$ and let
  $m_{\kappa_n} \in \mathcal{A}_{\kappa_n}$ BP-converge to
  $(R_0,r_0, a_0 ) \in \widetilde{\mathcal{A}_0}$ with
  \begin{align}
    \label{eq:Emknl0}
          \limsup_{n\to
          \infty}\frac{\mathcal{E}_{\kappa_n}(m_{\kappa_n}) -
          8\pi}{\kappa_n^2} < 0.
	\end{align}
	Then we have
	\begin{align}
          \liminf_{n\to
          \infty}\frac{\mathcal{E}_{\kappa_n}(m_{\kappa_n}) -
          8\pi}{\kappa_n^2}  \geq \mathcal{E}_0 (R_0,r_0, a_0 ) 
	\end{align}
and, in particular, $(R_0,r_0, a_0) \in \mathcal{A}_0$.
\end{proposition}

We next turn to the construction of a recovery sequence.  Essentially,
we take the Belavin-Polyakov profile determined by the limit problem
and modify the tail according to the harmonic function $u$ arising as
the minimizer of $T( a_0)$, see \eqref{eq:Ta0}.  The DMI term has
again been treated in Lemma \ref{lemma:lower_bound_DMI}, so that after
an appropriate construction only the Dirichlet term remains to be
analyzed.  To ensure that the tail correction does not affect the
skyrmion core, we modify $u$ to satisfy $u= 0$ in a small neighborhood
of $ a_0$, which is possible since points have zero capacity in
$H^1(\R^2)$.

\begin{proposition}\label{prop:upper_bound}
  For every $(R_0,r_0, a_0) \in \mathcal{A}_0$ and all sequences of
    $\kappa_n \to 0$ there exists a sequence
  $m_{\kappa_n} \in \mathcal{A}_{\kappa_n}$ BP-converging to
  $(R_0,r_0, a_0)$ such that
	\begin{align}
          \limsup_{n\to
          \infty}\frac{\mathcal{E}_{\kappa_n}(m_{\kappa_n}) -
          8\pi}{\kappa_n^2}  \leq \mathcal{E}_0 (R_0,r_0, a_0). 
	\end{align}
\end{proposition}

Throughout the proofs of these statements, we will omit the index $n$
from the notation by abuse of notation.

\begin{proof}[Proof of Lemma \ref{lemma:compactness}]
  Let $m_\kappa \in \mathcal A_\kappa$ be a sequence
  satisfying condition \eqref{eq:negative_energy}. We take
  $\phi_\kappa (x) = R_\kappa \Phi(\rho_\kappa^{-1}(x-{ a_\kappa }))$
  with $R_\kappa \in SO(3)$, $\rho_\kappa>0$ and
  ${ a_\kappa } \in \R^2$ to be a Belyavin-Polyakov profile achieving
  the Dirichlet distance of $m_\kappa$ to $\mathcal{B}$. We would like
  to show existence of a subsequence $m_\kappa$ (not relabelled)
  BP-converging to $(R_0, r_0, a_0) \in \widetilde{\mathcal A_0}$ (see
  Definition~\ref{def:topology}). Due to compactness of $SO(3)$ it is
  clear that there exists $R_0 \in SO(3)$ such that
  $R_\kappa \to R_0$, however, we need to show that
  $R_0 e_3=e_3$. Using Lemma~\ref{lemma:lower_bound_dirichlet}, we
  also know that $a_\kappa \in \Omega$ and hence
  $a_\kappa \to a_0 \in \overline\Omega$, so we only need to show
  that $a_0 \in \Omega$. We begin by estimating $\rho_\kappa$ and the
  Dirichlet distance through the Dirichlet excess.
	
  \textit{Step 1: Estimate
    $\operatorname{dist}({ a_\kappa },\R^2\setminus\Omega)$ and
    $Z(m_\kappa)$, and prove
		\begin{align}\label{eq:rho_approx_kappa}
                  0 < \liminf_{\kappa \to 0}
                  \frac{\rho_\kappa}{\kappa} \leq  \limsup_{\kappa \to
                  0} \frac{\rho_\kappa}{\kappa} < \infty. 
		\end{align}}
	
By our assumption \eqref{eq:negative_energy}, for $\kappa>0$ small
enough there exists a subsequence (not relabeled) and $C_1>0$, such
that we have
	\begin{align}\label{eq:excess_in_terms_of_DMI}
          Z(m_\kappa) -2\kappa \int_\Omega  m_\kappa' \cdot \nabla
          m_{\kappa,3}  \intd x = \mathcal{E}_\kappa(\m_\kappa) - 8\pi
          \leq - C_1 \kappa^2. 
	\end{align}
	Due to estimate \eqref{eq:lower_bound_DMI} from Lemma
        \ref{lemma:lower_bound_DMI}, there exists $C_2>0$ with
	\begin{align}
          Z(m_\kappa) - C_2 \kappa \left( \rho_\kappa +Z(m_\kappa) +
          \rho_\kappa^\frac{6}{5} \right) \leq -C_1 \kappa^2. 
	\end{align}
        Noting that $\rho_\kappa\leq C\kappa^{\frac12}$ by Lemma
        \ref{lemma:lower_bound_dirichlet}, if $\kappa$ is small enough
        we can absorb $Z(m_\kappa)$ and $\rho_\kappa^\frac{6}{5}$ in
        the second term on the left-hand side into the other terms,
        giving
	\begin{align}\label{eq:estzm}
          \frac{1}{2}Z(m_\kappa) - 2 C_2 \kappa  \rho_\kappa \leq -C_1
          \kappa^2. 
	\end{align}
	We may thus use Lemma \ref{lemma:lower_bound_dirichlet} again
        to obtain
	\begin{align}
          \frac{\rho_\kappa^2}{\operatorname{dist}^2 ({ a_\kappa} , \R^2
          \setminus \Omega) }  - 2C_2C \kappa  \rho_\kappa \leq -C_1 C
          \kappa^2. 
	\end{align}
	Completing the square on the left-hand side, we get
	\begin{align}\label{eq:complete_square}
          \left( \frac{\rho_\kappa }{\operatorname{dist}({ a_\kappa},
          \R^2\setminus \Omega)} - {C_2C}\kappa
          \operatorname{dist}({ a_\kappa }, \R^2\setminus \Omega)
          \right)^2 \leq \left( {C_2^2C^2}
          \operatorname{dist}^2 ({ a_\kappa}, \R^2\setminus
          \Omega) -C_1 C\right)  \kappa^2 
	\end{align}
        Since the left-hand must be non-negative and constants
          $C, C_1, C_2$ are positive, we obtain the estimate
	\begin{align}
          \liminf_{\kappa \to 0} \operatorname{dist}({ a_\kappa},
          \R^2\setminus \Omega) >0, 
	\end{align}
	so that after passing to a further subsequence we have
        $ a_\kappa \to a_0$ with $ a_0 \in \Omega$.
	
	Continuing from \eqref{eq:complete_square}, we also have
	\begin{align}
          \left| \frac{\rho_\kappa }{\kappa \operatorname{dist}({
          a_\kappa}, \R^2\setminus \Omega)} - {C_2C}
          \operatorname{dist}({ a_\kappa }, \R^2\setminus \Omega)
          \right| \leq \left( {C_2^2C^2} \operatorname{dist}^2({
          a_\kappa}, \R^2\setminus \Omega) -C_1 C\right)^\frac{1}{2}.
	\end{align}
	Consequently, we obtain
        $\limsup_{\kappa \to 0} \frac{\rho_\kappa}{\kappa} <\infty$.
        Since
	\begin{align}
          {C_2C}\operatorname{dist}({ a_\kappa }, \R^2\setminus
          \Omega) >  \left( {C_2^2C^2} \operatorname{dist}^2({
          a_\kappa}, \R^2\setminus \Omega) -C_1 C\right)^\frac{1}{2}, 
	\end{align}
        we also obtain
        $\liminf_{\kappa \to 0} \frac{\rho_\kappa}{\kappa} > 0$.
        Extracting a further subsequence, if necessary, we find
        $r\in (0,\infty)$ such that
        $\frac{\rho_\kappa}{\kappa} \to r$.
	
      By the estimate \eqref{eq:estzm}, Theorem \ref{thm:rigidity} and
      taking into account
      $\limsup_{\kappa \to 0} \frac{\rho_\kappa}{\kappa} <\infty$, we
      also get
      \begin{align}\label{eq:good_gradient_bound_proof}
        \limsup_{\kappa \to \infty} \kappa^{-2} \int_{\R^2} |\nabla
        (m_\kappa - \phi_\kappa)|^2 \intd x <\infty. 
      \end{align}

      \textit{Step 2: Prove $R_0 e_3 = e_3$.}

      As was already mentioned, the existence of $R_0 \in SO(3)$ such
      that $R_\kappa \to R_0$ along a subsequence simply follows from
      compactness of $SO(3)$. Therefore, we are left with showing
      $R_0 e_3 = e_3$.  By \cite[Lemma 2.5 and Lemma
      A.2]{bernand2019quantitative}, there is a constant $C>0$ such
      that
\begin{align}
  \int_{\R^2}\exp\left(\frac{2\pi}{3}\, \frac{| m_\kappa -
  \phi_\kappa| ^2}{\| \nabla (m_\kappa - \phi_\kappa)\|^2_{L^2(\R^2)}}
  \right) |\nabla \phi_\kappa|^2\intd x\leq C. 
\end{align}
Additionally, for $\kappa$ small enough, on
$\R^2\setminus B_{\operatorname{diam}(\Omega)} ({ a_\kappa}) \subset
\R^2\setminus \Omega$ we have $m_\kappa = -e_3$, as well as
$|\phi_\kappa + R_\kappa e_3| \leq C \kappa$ by Lemma
\ref{lemma:skyrmion_tail} and estimate \eqref{eq:rho_approx_kappa}.
As a result, also using estimate \eqref{eq:good_gradient_bound_proof},
for some constant $c>0$ we have 
\begin{align}
  \exp\left(c\frac{|R_\kappa e_3 - e_3|^2}{\kappa^2}  \right)
  \int_{\R^2\setminus B_{\operatorname{diam}(\Omega)} ({ a_\kappa})}
  |\nabla \phi_\kappa|^2 \intd x \leq C. 
\end{align}
The usual integration in polar coordinates therefore gives
\begin{align}
  \exp\left(c\frac{|R_\kappa e_3 - e_3|^2}{\kappa^2}  \right)
  \rho_\kappa^2 \leq C, 
\end{align}
which together with estimate \eqref{eq:rho_approx_kappa} implies
$\lim_{\kappa \to 0} R_\kappa e_3 =e_3$.
\end{proof}

\begin{proof}[Proof of Proposition \ref{prop:lower_bound}]
  Let $m_\kappa \in \mathcal A_\kappa$ BP-converge to
  $(R_0, r_0, a_0) \in \widetilde{\mathcal A_0}$.  We first choose a
  subsequence in $\kappa$ (not relabeled) such that
	\begin{align}
          \liminf_{\kappa \to 0} \frac{\mathcal{E}_\kappa(m_\kappa) -
          8\pi}{\kappa^2} = \lim_{\kappa \to
          0}\frac{\mathcal{E}_\kappa(m_\kappa) - 8\pi}{\kappa^2}, 
	\end{align}
	so that we may pass to further subsequences if necessary.
        Then by Lemma \ref{lemma:compactness}, and using the fact that
        BP-limits are unique, see Remark \ref{rem:unique}, we may
        further suppose that $R_\kappa \in SO(3)$, $\rho_\kappa$ and
        $ a_\kappa$ determine a Belavin-Polyakov profile
        $\phi_\kappa := R_\kappa \Phi( \rho_\kappa^{-1}(\bullet -
        a_\kappa))$ achieving the Dirichlet distance of $m_\kappa$ to
        $\mathcal{B}$.  In particular, we can apply Lemmas
        \ref{lemma:lower_bound_dirichlet}--\ref{lemma:lower_bound_DMI}.
	
	The fact that the DMI term converges to the corresponding
        expression in the $\Gamma$-limit follows immediately from
        Lemma \ref{lemma:lower_bound_DMI} and the assumptions
        \eqref{eq:m_close_to_phi} and \eqref{eq:radius_kappa} of
        BP-convergence, with an error of order
        $\kappa^{-1} Z(m_\kappa) + \kappa^\frac15$ as $\kappa \to 0$.
        We therefore only have to deal with the limit behavior of the
        Dirichlet energy excess, which satisfies
        $Z(m_\kappa) \leq C \kappa^2$ by \eqref{eq:Emknl0} for
        $\kappa$ small enough.
	
	To this end, we note that
        $\phi_\kappa :\R^2 \to \Sphere^2 \setminus\{-R_\kappa e_3\}$
        is a bijective, conformal mapping.  Therefore we may introduce
        the function $v_\kappa : \Sphere^2 \to \R^3$ defined as
        $v_\kappa := \rho_\kappa^{-1}( m_\kappa \circ \,
        \phi_\kappa^{-1} -\operatorname{id}_{\Sphere^2} )$.  From
        assumptions \eqref{eq:m_close_to_phi} and
        \eqref{eq:radius_kappa}, and the change of variables formula
        \cite[Lemma A.2]{bernand2019quantitative} for conformal
        mappings we get
        	\begin{align}
                  \limsup_{\kappa \to 0} \int_{\Sphere^2}
                  |\nabla v_\kappa|^2 \intd \mathcal H^2(z) & <
                                                              \infty.  
        	\end{align}
                Thus, applying the Poincar{\'e}-type estimate
                \cite[Lemma 2.5]{bernand2019quantitative} to
                $\rho_\kappa v_\kappa$ we get 
	\begin{align}
          \limsup_{\kappa \to 0} \int_{\Sphere^2} \left(  |v_\kappa
          |^2 + |\nabla v_\kappa|^2 \right) \intd \mathcal H^2(z)
          & < \infty.  
	\end{align}
	Consequently, up to a subsequence there exists
        $v_0 \in H^1(\Sphere^2; \R^3)$ such that
        $v_\kappa \rightharpoonup v_0$ weakly in
        $H^1(\Sphere^2; \R^3)$.
	
	For $z\in \Sphere^2$, we compute
	\begin{align}
	  \begin{split}
            v_\kappa (z)\cdot z & =\rho_\kappa ^{-1} (m_\kappa \circ
            \, \phi_\kappa^{-1}(z) -z)\cdot z = \rho_\kappa ^{-1}
            (m_\kappa \circ \, \phi_\kappa^{-1}(z)\cdot z -1 ) =
            -\frac{1}{2\rho_\kappa} | m_\kappa \circ \,
            \phi_\kappa^{-1}(z)- z|^2.
	  \end{split}
	\end{align}
	Therefore, another application of \cite[Lemma
        2.5]{bernand2019quantitative} and assumptions
        \eqref{eq:m_close_to_phi} and \eqref{eq:radius_kappa} gives
	\begin{align}
          \limsup_{\kappa \to 0} \kappa^{-2}\int_{\Sphere^2} |v_\kappa
          (z)\cdot z |^2 \intd \mathcal H^2(z) = \limsup_{\kappa \to 0}
          \int_{\Sphere^2}  \frac{1}{4\rho_\kappa^2\kappa^2} |
          m_\kappa \circ \, \phi_\kappa^{-1}(z)-  z|^4 \intd
          \mathcal H^2(z)   <
          \infty. 
	\end{align}
	As a result, in the limit $\kappa \to 0$, we get that
	\begin{align}\label{eq:tangent}
          v_0(z) \cdot z = 0 \text{ for }
          \mathcal{H}^2 \text{-a.e. }z\in \Sphere^2. 
	\end{align}
	In particular, $v_0$ is an $H^1$-regular, tangent vector field
        on the sphere.
	
	We define a set
        $U_\kappa:= \phi_\kappa ( B_{\sqrt{\rho_\kappa}}( a_\kappa))
        \subset \Sphere^2$ and a function
        $w_\kappa := m_\kappa +e_3 - \phi_\kappa - R_\kappa e_3$.  By
        \cite[Lemma A.4]{bernand2019quantitative}, see also
        \cite[Lemma 9]{li2018stability}, the excess can be rewritten
        as
	\begin{align}
	  \begin{split}
            \kappa^{-2}Z(m_\kappa) & = \kappa^{-2} \left( \int_{\R^2}
              |\nabla (m_\kappa-\phi_\kappa)|^2 \intd x - \int_{\R^2}
              |m_\kappa - \phi_\kappa|^2 |\nabla \phi_\kappa|^2\intd
              x\right)\\
            & = \frac{\rho_\kappa^2 }{\kappa^2} \int_{\R^2 \setminus
              B_{\sqrt{\rho_\kappa}}( a_\kappa)} \left|\nabla \left(
                \rho_\kappa^{-1} w_\kappa \right) \right|^2 \intd x
            +\frac{\rho_\kappa^2}{\kappa^2} \left( \int_{U_\kappa} |
              \nabla v_\kappa|^2 \intd \mathcal H^2(z) -2
              \int_{\Sphere^2} | v_\kappa|^2 \intd \mathcal H^2(z)
            \right),
	  \end{split}
	\end{align}
	where the transformation of the expressions onto the sphere is
        again via \cite[Lemma A.2]{bernand2019quantitative}.
                
        Due to assumptions \eqref{eq:m_close_to_phi} and
        Lemma \ref{lemma:skyrmion_tail}, there exists
        $u \in H_\mathrm{loc}^1(\R^2;\R^2)$ such that
        \begin{align}        	
          \rho_\kappa^{-1} w'_\kappa \rightharpoonup 
          u
        \end{align} in
        $H^1_{\mathrm{loc}}(\R^2; \R^2)$ as $\kappa \to 0$.
        Let $R_0' \in SO(2)$ be such that $R_0 v = (R_0' v', v_3)$ for all
        $v \in \R^3$. Note that, again by Lemma
        \ref{lemma:skyrmion_tail} and assumption
        \eqref{eq:radius_kappa}, the limit 
        for $x\in \R^2\setminus \Omega$ satisfies 
       	\begin{align}\label{eq:boundary_data}
          u(x) = 2 R_0' \frac{x-a_0}{|x-a_0|^2}.
       	\end{align}
        
        Let $\delta \in (0, \frac{1}{2})$. For $\kappa > 0$ small
        enough we have by BP-convergence that
        $B_{\sqrt{\rho_\kappa}}( a_\kappa) \subset B_{\delta}( a_0)$.
        Similarly, due to \eqref{eq:m_close_to_phi} and the definition
        of $\Phi$ (see \eqref{def:BPprofile}), for the set
        $V_\delta := \{z\in \Sphere^2: |z+e_3| >\delta\}$ we obtain
        $V_\delta \subset (R_\kappa \Phi) (B_{\rho_\kappa^{-1/2}}(0))
        = U_\kappa $. Therefore, for any fixed $\tilde r> \delta$ and
        $\kappa$ small enough, we obtain
        \begin{align}
          \kappa^{-2}Z(m_\kappa) \geq \frac{\rho_\kappa^2 }{\kappa^2}
          \int_{B_{\tilde r} (a_0) \setminus 
          B_{\delta}( a_0)} \left|\nabla \left(
          \rho_\kappa^{-1} w'_\kappa \right) \right|^2 \intd x
          +\frac{\rho_\kappa^2}{\kappa^2} \left( \int_{V_\delta} |
          \nabla v_\kappa|^2 \intd \mathcal H^2(z) -2
          \int_{\Sphere^2} | 
          v_\kappa|^2 \intd \mathcal H^2(z) \right).
	\end{align}
        Together with the compact Sobolev embedding
        $H^1(\Sphere^2) \hookrightarrow L^2(\Sphere^2)$, we therefore
        have in the limit $\kappa \to 0$ that
	\begin{align}
		\begin{split}
                  \liminf_{\kappa \to 0} \kappa^{-2}Z(m_\kappa) & \geq
                  r^2 \int_{B_{\tilde r}(a_0) \setminus B_{\delta}(
                    a_0)} |\nabla u |^2 \intd x +r^2\left(
                    \int_{V_\delta} | \nabla v_0|^2 \intd \mathcal
                    H^2(z) -2\int_{\Sphere^2} | v_0|^2 \intd \mathcal
                    H^2(z) \right).
		\end{split}
	\end{align}
	Letting $\delta \to 0$ and $\tilde r \to \infty$, we consequently get 
	\begin{align}
	  \begin{split}
            \liminf_{\kappa \to 0} \kappa^{-2}Z(m_\kappa) & \geq r^2
            \int_{\R^2} |\nabla u |^2 \intd x +r^2\left(
              \int_{\Sphere^2} | \nabla v_0|^2 \intd \mathcal H^2(z)
              -2\int_{\Sphere^2} | v_0|^2 \intd \mathcal H^2(z)
            \right).
	  \end{split}
	\end{align}
	The non-negativity of the second variation of the Dirichlet
        energy at minimizers for tangent vector fields on the sphere
        \cite[(4.5)]{bernand2019quantitative} finally implies
	\begin{align}
		\begin{split}
                  \liminf_{\kappa \to 0} \kappa^{-2}Z(m_\kappa) & \geq
                  r^2 \int_{\R^2} |\nabla u |^2 \intd x = r^2
                  \int_{\R^2} |\nabla \tilde u|^2 \intd x \geq r^2 T(
                  a_0),
		\end{split}
	\end{align}
        where we noted that $\tilde u:= (R_0')^{-1} u$ satisfies the
        boundary data required in the definition of $T$, see identity
        \eqref{eq:boundary_data}.
\end{proof}

\begin{proof}[Proof of Proposition \ref{prop:upper_bound}]
  In contrast to the rest of the paper, in this proof the constant $C$
  may depend on $r_0$ and $ a_0$.  We fix
  $(R_0, r_0, a_0) \in \mathcal A_0$ and take
  $u \in \mathring H^1(\R^2;\R^2)\cap L^2_{\mathrm{loc}}(\R^2;\R^2)$
  with $u(x) = 2\frac{x- a_0}{|x- a_0|^2}$ for
  $x\in \R^2\setminus \Omega$ and achieving
	\begin{align}
		\int_{\R^2}  |\nabla u|^2 \intd x = T( a_0).
	\end{align}
	In particular, $u$ is harmonic in $\Omega$, unique, and by the
        maximum principle satisfies $u \in L^\infty(\R^2; \R^2)$.

        We start by constructing two auxiliary functions which will be
        useful in the construction of the recovery sequence.

\textit{Step 1: Truncating $u$ at $ a_0$.}

For
$\delta \in \left(0,\frac{1}{2} \wedge
  \frac{1}{2}\operatorname{dist}^2( a_0, \R^2\setminus \Omega)\right)$ 
and $x\in \R^2$, we define
	\begin{align}
          \eta_\delta (x):= \begin{cases}
            1 & \text{ if } |x -  a_0| \leq \delta, \\
            \frac{2 \log|x- a_0|}{\log \delta } - 1 & \text{ if }
            \delta <|x -  a_0| \leq \delta^\frac{1}{2},\\ 
            0 & \text{ else.}
		\end{cases}
	\end{align}
	This function satisfies $\eta_\delta \in H^1(\R^2)$,
        $\operatorname{supp} \eta_\delta \subset B_{\sqrt{\delta}}( a_0)
        \subset \Omega$ and
	\begin{align}\label{eq:capacity}
		\int_{\R^2} |\nabla \eta_\delta|^2 \intd x \leq
          \frac{C}{|\log \delta|}.
	\end{align}
	We now set $u_\delta:= (1-\eta_\delta)u$ in order to enforce
        $u_\delta = 0$ in $B_{\delta}( a_0)$.  Then we still have
        $u_\delta \in \mathring H^1(\R^2; \R^2) \cap L^\infty(\R^2;
        \R^2)$, $u_\delta (x) = 2\frac{x- a_0}{|x- a_0|^2}$ for
        $x\in \R^2\setminus \Omega$ and
	\begin{align}\label{eq:dense_in_energy}
	  \begin{split}
          \int_{\R^2} |\nabla u_\delta|^2 \intd x - \int_{\R^2}
          |\nabla u|^2 \intd x & = \int_{B_{\sqrt{\delta}}( a_0)}
          \left(|- u \otimes \nabla \eta_\delta + (1-\eta_\delta)
          \nabla u|^2 - |\nabla u|^2 \right) \intd x\\
          & \leq C \int_{B_{\sqrt{\delta}}( a_0)}
          \left(| \nabla \eta_\delta|^2 + |\nabla u|^2 \right) \intd x \\
          & = o_\delta(1) 
      \end{split}
	\end{align}
	by the estimate \eqref{eq:capacity} and $|\nabla u| \in
        L^2(\R^2)$.
	
        \textit{Step 2: Construct the boundary data corrector
          $v_\kappa: \R^2 \to \R^2$ with $v_\kappa(x) = 0$ for all
          $x \in B_{ \operatorname{dist}(a_0,\partial \Omega)/2}(a_0)$
          and
	\begin{align}\label{eq:estimate_correction}
          \|v_\kappa\|_{L^\infty(\R^2)}  +\|\nabla
          v_\kappa\|_{L^2(\R^2)}  \leq C\kappa^2. 
	\end{align}}
      
      Let $\phi_\kappa := R_0 \Phi((r_0\kappa) ^{-1} (\bullet - a_0))$
      and again let $R_0'\in SO(2)$ be such that $R_0v = (R'_0v',v_3)$
      for all $v \in \R^3$.  In order to achieve the correct boundary
      data, we define
      $v_\kappa(x):= - \phi'_\kappa(x) - r_0\kappa R'_0 u$ for
      $x\in \R^2\setminus \Omega$ and $v_\kappa = 0$ in
      $B_{ \operatorname{dist}(a_0,\partial \Omega)/2}(a_0)$.
      Exploiting the estimates \eqref{eq:convergence_tail_exterior}
      and \eqref{eq:convergence_tail_gradient_exterior}, we can extend
      $v_\kappa(x)$ using \cite[Theorem 3.1]{evans} to a Lipschitz
      function on $\R^2$ such that
     \begin{align}
       \|v_\kappa\|_{W^{1,\infty}(\R^2)} \leq C
       \kappa^2.
     \end{align}
     The $L^2$ estimate for the gradient on the whole space follows.

\textit{Step 3: Definition of the recovery sequence.}

Having introduced the boundary corrector $v_\kappa$, we may now define
the test magnetizations
$m_{\kappa,\delta} \in \mathring H^1(\R^2;\Sphere^2)$ as follows:
	\begin{align}
          m_{\kappa,\delta} (x) :=
		\begin{cases}
                  \phi_\kappa(x) & \text{ if } x \in B_\delta( a_0 ),\\
                  p\left( \phi'_\kappa + r_0\kappa R'_0
                    u_\delta+v_\kappa\right) & \text{ if } x \in
                  \R^2\setminus B_\delta( a_0),
		\end{cases} 
	\end{align}
	where the map $p: \R^2 \to \Sphere^2$ lifts
        $v \in B_1(0) \subset \R^2$ to $\Sphere^2$ via
        $p(v) :=\left(v,-\sqrt{1-v^2}\right)$.  It is clear that since
        $u$ and $u_\delta$ coincide outside $B_{\sqrt{\delta}}(a_0)$,
        using the definition of $v_\kappa$, we have
        $m_{\kappa,\delta} (x)=-e_3$ for all
        $x \in \R^2 \setminus \Omega$. Furthermore, for small enough
        $\delta>0$ we have $u_\delta=0$ and $v_\kappa=0$ in
        $\overline{B_\delta}(a_0)$ and therefore the test
        configuration $m_{\kappa,\delta} $ is well defined for all
        $\kappa$ sufficiently small depending on $\delta$.  We also
        have
          \begin{align} \label{eq:smallness} \| \phi'_\kappa+
            r_0\kappa R_0' u_\delta +v_\kappa\|_{L^\infty(\R^2
              \setminus B_\delta(a_0))} \leq C_\delta \kappa,
    \end{align}
    by the definition of $\Phi'$, boundedness of $u_\delta$, and
    estimate \eqref{eq:estimate_correction}. Here and in the
    following, the symbol $C_\delta>0$ denotes a generic positive
    constant depending only on $\Omega$, $r_0$, $a_0$, and $\delta$.

    Let
    $q_{\kappa,\delta} (x) := m_{\kappa,\delta}(x) -
    \phi_\kappa(x) - r_0\kappa R_0 (u_\delta(x),0)$.  For
    $x\in B_\delta(a_0)$ we of course have $q_{\kappa,\delta} = 0$.
    For $x\in \R^2\setminus B_\delta(a_0)$ we compute
    \begin{align}
      q_{\kappa,\delta}(x) = \left(v_\kappa, p_3\left(
      \phi'_\kappa + r_0\kappa R'_0
      u_\delta+v_\kappa\right)  - p_3\left( \phi'_\kappa
      \right)  \right). 
    \end{align}
    Using estimates \eqref{eq:estimate_correction} and
    \eqref{eq:smallness}, as well as Lipschitz continuity of the
    square root near 1, we get
    \begin{align}
      \|q_{\kappa,\delta}\|_{L^\infty(\R^2)} 
      \leq C\left(\kappa^2 +  \|\phi'_\kappa + r_0\kappa R'_0
      u_\delta+v_\kappa\|_{L^\infty(\R^2\setminus
      B_\delta(a_0))}^2 +\|\phi'_\kappa\|_{L^\infty(\R^2\setminus
      B_\delta(a_0)}^2  \right)  
      \leq C_\delta\kappa^2. 
    \end{align}
    To estimate the $H^1$-norm of $q_{\kappa,\delta}$, note that for
    any $v \in H^1(\R^2)$ with $\| v \|_{L^\infty(\R^2)} < 1$ we have
    $\nabla p_3(v) = \frac{v}{\sqrt{1-v^2}} \nabla v$ a.e. in $\R^2$
    by the weak chain rule.  Therefore, arguing as in
    \cite[(A.67)]{bernand2019quantitative} and using
    \eqref{eq:dense_in_energy}, \eqref{eq:estimate_correction} and
    \eqref{eq:smallness}, for all $\kappa$ small enough we get
	\begin{align}
          \left\| \nabla \left( p_3(\phi'_\kappa) \right)
          \right\|_{L^2(\R^2\setminus B_\delta(a_0))}
          & \leq
            C_\delta\kappa^2,\\ 
          \left\| \nabla \left( p_3\left( \phi'_\kappa + r_0\kappa
          R'_0 u_\delta+v_\kappa\right)\right)
          \right\|_{L^2(\R^2\setminus B_\delta(a_0))}
          & \leq  C_\delta\kappa^2,
	\end{align} 
	so that again with \eqref{eq:estimate_correction} we have
	\begin{align}\label{eq:estimate_remainder}
		\| \nabla q_{\kappa,\delta}\|_{L^2(\R^2)} \leq  C_\delta\kappa^2.
	\end{align}
        In particular, by the definition of $q_{\kappa, \delta}$ and
        \eqref{eq:dense_in_energy} we have the estimate
	\begin{align}\label{eq:construction_close_to_BP}
          \int_{\R^2} |\nabla (m_{\kappa,\delta} -\phi_{\kappa}) |^2
          \intd x \leq C_\delta \kappa^2,
	\end{align}
        for all $\kappa$ sufficiently small depending on $\delta$.  In
        particular, by the definition of
        $\mathcal{N}(m_{\kappa, \delta})$ in \eqref{eq:N} and the fact
        that $\mathcal{N}(m_{\kappa, \delta}) \in \mathbb Z$ we have
        $\mathcal{N}(m_{\kappa, \delta}) = 1$ and, therefore,
        $m_{\kappa,\delta} \in \mathcal{A}$ for small enough $\kappa$.
	
\textit{Step 4: Computation of the energy.}

By Lemma \ref{lemma:lower_bound_DMI} and estimate
\eqref{eq:construction_close_to_BP}, we have
	\begin{align}
		\begin{split}
                  & \quad \left|-2\kappa \int_\Omega
                    \m'_{\kappa,\delta}\cdot \nabla
                    m_{\kappa,\delta,3} \intd x + 2 \kappa \rho
                    \int_{\R^2} (R_0 (\Phi+e_3))'\cdot \nabla
                    \Phi_3  \intd x \right| \\
                  & \leq C_\delta \kappa^\frac{11}{5}.
		\end{split}
	\end{align}
	For the Dirichlet energy, we again use \cite[Lemma
        A.4]{bernand2019quantitative} to get
	\begin{align}
	  \begin{split}
            \int_{\R^2} |\nabla m_{\kappa,\delta} |^2 \intd x - 8\pi
            & = \int_{\R^2} |\nabla (m_{\kappa,\delta} -\phi_{\kappa})
            |^2 \intd x - \int_{\R^2} |m_{\kappa,\delta}
            -\phi_{\kappa}|^2|\nabla \phi_{\kappa} |^2 \intd x\\ 
            & \leq \int_{\R^2} |\nabla (m_{\kappa,\delta}
            -\phi_{\kappa}) |^2 \intd x .
	  \end{split}
	\end{align}
	By the estimates \eqref{eq:estimate_remainder} and
        \eqref{eq:dense_in_energy}, we get
	\begin{align}
	  \begin{split} 
            \int_{\R^2} |\nabla m_{\kappa,\delta} |^2 \intd x - 8 \pi
            & \leq r_0^2\kappa^2 \int_{\R^2} |\nabla u_\delta|^2 \intd
            x + C_\delta \kappa^3  \\
            & \leq r_0^2\kappa^2 \int_{\R^2} |\nabla u|^2 \intd x +
            C_\delta \kappa^3 + \kappa^2 o_\delta(1).
	  \end{split}
	\end{align}
	Therefore, we have
	\begin{align}
         \limsup_{\kappa\to 0}
          \frac{\mathcal{E}_\kappa(m_{\kappa,\delta}) -
          8\pi}{\kappa^2} \leq \mathcal{E}_0(R_0,r_0,a_0) + o_\delta(1). 
	\end{align}
	By a diagonal argument, the statement then follows.
      \end{proof}

        \begin{proof}[Proof of Theorem \ref{thm:convergence}]
          After we established the compactness of sequences obeying
          \eqref{eq:negative_energy} with respect to the
          BP-convergence in Lemma \ref{lemma:compactness}, the
          statement of Theorem \ref{thm:convergence} follows by
          combining Propositions \ref{prop:lower_bound} and
          \ref{prop:upper_bound}, and noting that by Proposition
          \ref{prop:lower_bound} the limit of the sequence in Lemma
          \ref{lemma:compactness} belongs to $\mathcal A_0$.
        \end{proof}

  
\section{Analyzing the limit problem}
\label{sec:lim}

\begin{proof}[Proof of Theorem \ref{thm:minimizers}]
  By the properties of $\Gamma$-convergence, minimizers $m_\kappa$ of
  $\mathcal{E}_\kappa$ BP-converge to minimizers
  $(R_0,r_0, a_0) \in \mathcal{A}_0$ of $\mathcal{E}_0$ as
  $\kappa \to 0$ with the rate
	\begin{align}
          \int_{\R^2} |\nabla (m_\kappa -\phi_\kappa)|^2 \intd x \leq
          C \kappa^2, 
	\end{align}
	where
        $\phi_\kappa:= R_0\Phi\left(\frac{\bullet \ - \ a_0}{\kappa
            r_0} \right)$.  Note that Theorem \ref{thm:convergence}
        does apply to minimizers of $\mathcal E_\kappa$ over
        $\mathcal A$ in view of Lemma \ref{lemma:construction}.
	
	Recall that
	\begin{align}
          \mathcal{E}_0 (R_0,r_0, a_0) = r_0^2 T( a_0) - 2 r_0
          \int_{\R^2} (R_0\Phi)' \cdot \nabla \Phi_3 \intd x. 
	\end{align}
        Since by the Cauchy-Schwarz inequality the integrand of the
        DMI term is minimized at each point when $R'\Phi'$ and
        $\nabla \Phi_3$ are parallel and since $\Phi'$ is parallel
          to $\nabla \Phi_3$, the DMI term as a whole is minimized for
          $R_0= \operatorname{id}$.  Direct calculation or
        \cite[Lemma A.5]{bernand2019quantitative} gives
 	\begin{align}
          2 \int_{\R^2} \Phi' \cdot \nabla \Phi_3 \intd x = 8\pi.
 	\end{align}
      For
        $ a_0 \in \operatorname{arg min}_{a\in \Omega} T(a)$
        minimizing in $r_0$ therefore gives
 	\begin{align}
 		r_0& = \frac{4\pi}{T( a_0)},\\
 		\mathcal{E}_0 (R_0,r_0, a_0) & = - \frac{16\pi^2}{T( a_0)}.
 	\end{align}
        which completes the proof.
\end{proof}

\begin{proof}[Proof of Proposition \ref{lemma:representation}]
  First of all, observe that under our assumptions, \eqref{eq:ez0} is
  uniquely solvable in
  $C^\infty(\Omega; \mathbb C) \cap C^{1,\alpha}(\overline \Omega ;
  \mathbb C)$, see \cite[Theorem 8.34]{gilbarg}. The expression in
  \eqref{eq:Tuz} may be conveniently rewritten as an integral over
  $\partial \Omega$. Integrating by parts and using that $u_{z_0}$ is
  harmonic both inside and outside $\Omega$ (it is anti-holomorphic in
  $\Omega^c$), we obtain
\begin{align}
  \label{eq:Tu}
  \begin{split}
  T(z_0)
  & =  \int_\Omega \nabla \cdot ( \bar u_{z_0} \nabla u_{z_0} )  \,
    \intd x + \int_{\Omega^c} \nabla \cdot ( \bar u_{z_0} \nabla
    u_{z_0} )  \,  
    \intd x  \\
  & = \int_{\partial \Omega} \bar u_{z_0} \left( \left. \partial_\nu u_{z_0}
    \right|_{\Omega} - \left. \partial_\nu u_{z_0}
    \right|_{\Omega^c} \right) \intd \mathcal H^1(z),
  \end{split}
\end{align}
where $\partial_\nu$ denotes the derivative in the direction of the
outward unit normal to $\partial \Omega$.

Since $\Omega$ is simply connected, both the real and the imaginary
parts of the harmonic function $u_{z_0}$ possess the unique, up to
constants, harmonic conjugates in $\overline \Omega$ that belong to
$C^{1,\alpha}(\overline \Omega)$. Therefore $u_{z_0}$ admits a
decomposition
\begin{align}
  \label{eq:fg}
  u_{z_0}(z) = f(z) + \overline{g(z)} \qquad z \in \overline{\Omega},
\end{align}
for two functions $f(z)$ and $g(z)$ that are holomorphic in
$\overline \Omega$.

Recall that if $\nu : \partial \Omega \to \mathbb C$ represents the
outward unit normal to $\partial \Omega$ and $f$ is holomorphic in
$\overline \Omega$, we have $\partial_\nu f = \nu f'$ on
$\partial \Omega$, where the prime denotes the usual derivative of a
holomorphic function. Similarly, if $\tau := i \nu$ represents the
unit tangent to $\partial \Omega$ in the counter-clockwise direction,
the tangential derivative $\partial_\tau f = \tau f'$ on
$\partial \Omega$. Therefore, using \eqref{eq:fg} we can write
\begin{align}
  \left. \partial_\nu u_{z_0} \right|_\Omega = \nu f' +
  \overline{\nu g'}, \qquad \left. \partial_\nu u_{z_0} \right|_{\Omega^c}
  = -{2 \bar \nu \over (\bar z - \bar z_0)^2}
\end{align}
on $\partial \Omega$.  At the same time, by continuity of $u_{z_0}$
across $\partial \Omega$ we have
\begin{align}
  \tau f' + \overline{\tau g'} = -{2 \bar \tau \over (\bar z - \bar z_0)^2}
\end{align}
on $\partial \Omega$. Thus we have
\begin{align}
  \left. \partial_\nu u_{z_0} \right|_\Omega - \left. \partial_\nu u_{z_0}
  \right|_{\Omega^c} = -2 i \tau f'
\end{align}
on $\partial \Omega$, and the integral in \eqref{eq:Tu} can be
rewritten as a Cauchy type contour integral
\begin{align}
  T(z_0) = - 4 i \oint_{\partial \Omega} {f'(z) \over z - z_0} \intd
  z. 
\end{align}
Finally, applying the residue theorem, we obtain
$T(z_0) = 8 \pi f'(z_0)$, which is precisely \eqref{eq:TzW}.
\end{proof}

\subsection{Disks}

\begin{proof}[Proof of Proposition \ref{prop:balls}]
  Without loss of generality, we may assume $\ell = 1$. Let $u_{z_0}$
  be defined by \eqref{def:minimizer_ball}. We recall that for
  $z \in \mathbb C \backslash B_1(0)$ we have
  \begin{align}
    u_{z_0}(z) = \frac{2}{\bar z - \bar z_0}.
  \end{align}
  In particular, up to complex conjugation, $u_{z_0}$ is invariant
  under the Kelvin transform, i.e., for all
  $z \in \overline{B_1(0)}$ we have
  \begin{align}
    u_{ z_0}\left( \bar z^{-1} \right) =  \frac{2 z}{1 - \bar  z_0 z}
    =   u_{z_0}(z),
  \end{align}
  which is holomorphic and, therefore, harmonic in
  $B_1(0)$. Furthermore, the function $u_{ z_0}$ is continuous
  across $\partial B_1(0)$.  By the uniqueness of boundary value
  problems for harmonic functions and continuity at the boundary,
  $u_{z_0}$ is indeed the function achieving $T(z_0)$.  By
    Proposition \ref{lemma:representation}, we obtain
  \begin{align}
    T(z_0) = \frac{16 \pi}{(1 - |z_0|^2)^2}.
  \end{align}  
  Clearly, this expression is minimized for $z_0 = 0$. The rest of the
  statement is obtained by a direct substitution.
\end{proof}

\subsection{Strips}\label{sec:proof_stripes}

Instead of giving a full proof for the $\Gamma$-convergence in the
case of a strip $\Omega_\ell := \R \times (-\ell/2,\ell/2)$ for
$\ell>0$, we point out the places in which the proof for bounded sets
needs to be adjusted.

The first adjustment concerns Lemma \ref{lemma:a_priori}, where a
Poincar{\'e} estimate still holds for all admissible $m$.
Furthermore, the proof of existence of a minimizer follows the lines
of the proof in the whole space \cite{bernand2019quantitative} to deal
with the non-compact invariance under horizontal shifts.

Step 1 of Lemma \ref{lemma:lower_bound_dirichlet} applies verbatim.
Lemma \ref{lemma:skyrmion_tail} also works similarly, one just needs
to use the decay behavior of the leading order contribution of $w_m$
when applying a Poincar{\'e} type estimate on slices
$\{x_1\}\times (-\ell,\ell)$ in order to achieve finite $L^2$ norm of
$w_m$ on $\Omega_\ell$.  In Lemma \ref{lemma:lower_bound_DMI},
additional care needs to be taken in the integration
\eqref{eq:L2_of_BP}, i.e.,
	\begin{align}
	\int_{\rho^{-1}(\Omega_\ell-{ a })} |\Phi + e_3|^2
	\intd x  \leq C |\log \rho|,
\end{align}
although the result is unchanged.  Furthermore, the Sobolev embedding
$H^1(\Omega_\ell) \hookrightarrow L^5(\Omega_\ell)$ still works by
virtue of $\Omega_\ell$ being an extension domain and \cite[Theorem
8.5(ii)]{lieb2001analysis}.

The remaining arguments work the same, up to the adjustment that due
to translation invariance the component $(a_\kappa)_1$ cannot be
controlled.

\begin{proof}[Proof of Proposition \ref{prop:stripes}]

  Again, without loss of generality we may assume $\ell=\pi/2$. Let
  $\Omega :=\R\times (-\frac\pi 4,\frac\pi 4)$, and let $u_{y_0}$ be
  the map defined in equation \eqref{def:minimizer_stripes} for
  $\ell=\pi/2$ and $y_0\in (-\frac\pi 4,\frac\pi 4)$.  The fact that
  $u_{y_0}$ satisfies the boundary conditions follows from the two
  elementary identities:
    \begin{align}
      \tanh\left( z \pm \frac{i \pi}{4}\right) = \coth\left( z \mp
      \frac{i \pi}{4}\right)  \quad \text{for all $z \in \mathbb{C}$.}
    \end{align}	
    Furthermore, as $u_{y_0}$ is a sum of a holomorphic and an
    anti-holomorphic functions in $\Omega$, it is harmonic in all the
    points where it is finite. As is well-known, the only
    singularities of both $\tanh z$ and $\coth z$ are simple poles on
    the imaginary axis.  Since
    $|\operatorname{Im}(z+iy_0)| < \frac{\pi}{2} $ for all
    $z \in \Omega$, the $\tanh$ contribution is smooth in $\Omega$.
    For the same reason, the $\coth$ contribution only has a
    singularity at $z= iy_0$, which, however, is precisely
    counterbalanced by $\frac{2}{\bar z +iy_0}$, as can be seen from
    the Laurent series of $\coth$ at the origin.  Therefore $u_{y_0}$
    is indeed the map achieving $T(iy_0)$.  As $u_{y_0}$ decays
      sufficiently quickly at infinity, the arguments leading to
      Proposition \ref{lemma:representation} may also be adapted to
      the setting of strips, whereby we have
	\begin{align}
		\label{energyComputation}
          T(iy_0)=\frac{16\pi}{\cos^2(2y_0)}.
	\end{align}
        This function is clearly minimized by $y_0=0$, giving the
        statement.
\end{proof}

\subsection{Half-plane}

\begin{proof}[Proof of Proposition \ref{prop:half_space}]
	
  Checking that $u_{y_0}$ is indeed the minimizer realizing $T(i y_0)$
  is trivial, and the rest of the statement is obtained via
  Proposition \ref{lemma:representation}, again, adapted to the
  half-plane setting.
\end{proof}

\section{Anisotropy as a continuous perturbation}
\label{sec:anicont}

\begin{proof}[Proof of Proposition \ref{lemma:anisotropy}]
  Due to the properties of the BP-convergence, there exist sequences
  $R_n \in SO(3)$, $\rho_n>0$ and $a_n \in \Omega$ such that with
  $\phi_n (x) := R_n\Phi(\rho_n^{-1}(x-a_n))$ for $x\in \R^2$ the
  estimates of Definition \ref{def:topology} hold. Again, by
  Friedrichs' inequality \cite[Corollary 6.11.2]{mazya} we have
  \begin{align}
    \label{eq:Frid}
		 \begin{split}
                   & \quad \int_\Omega \left| m_{\kappa_n} + e_3 -
                     \phi_n - R_n e_3  \right|^2 \intd x\\
                   & \leq C\left(\int_\Omega \left|\nabla (
                       m_{\kappa_n} - \phi_n) \right|^2 \intd x +
                     \int_{\partial B_{\mathrm{diam}(\Omega)}(a_0)}
                     \left| \phi_n + R_n e_3 \right|^2 \intd \mathcal
                     H^1(x) \right),
		\end{split}
		\end{align}
                which allows to control the $L^2$--distance between
                $m_{\kappa_n}$ and the Belavin-Polyakov profile
                $\phi_n$ that approximates it. In particular, by Lemma
                \ref{lemma:skyrmion_tail} and the properties of the
                BP-convergence, the two error terms in the right-hand
                side of \eqref{eq:Frid} are of order $\kappa_n^2$ for
                all $n$ large enough. Therefore, additionally
                reparametrizing the integral by the factor
                $\rho_n^{-1}$ in the second step and using the
                assumption
                $\lim_{n\to \infty}\frac{\rho_n}{\kappa_n} = r_0$, we
                have that
	\begin{align}
		\begin{split}
                  \lim_{n\to \infty} \frac{1}{\kappa_n^2 |\log
                    \kappa_n|} \int_{\Omega} |m'_{\kappa_n}|^2 \intd x
                  & =  \lim_{n\to \infty} \frac{1}{\kappa_n^2 |\log
                    \kappa_n|} \int_{\Omega}\left|\left((R_n
                      (\Phi(\rho_n^{-1}(x-a_n)) +
                      e_3)\right)'\right|^2 \intd x\\ 
                  & = \lim_{n\to \infty} \frac{r_0^2}{|\log \kappa_n|}
                  \int_{\rho_n^{-1}(\Omega-a_n)}\left|\left((R_n
                      (\Phi(x) + e_3)\right)'\right|^2 \intd x.
		\end{split}
	\end{align}

	Since $\Phi_3+1 \in L^2(\R^2)$, the contribution of
        $R_n(0, \Phi_3+1)$ in the last integral is negligible, so that
        by expanding the square we get
	\begin{align}\label{eq:reduced_expression}
		\begin{split}
                  & \quad \lim_{n\to \infty} \frac{r_0^2}{|\log
                    \kappa_n|}
                  \int_{\rho_n^{-1}(\Omega-a_n)}\left|\left((R_n
                      (\Phi(x) + e_3)\right)'\right|^2 \intd x\\ &
                  =\lim_{n\to \infty} \frac{r_0^2}{|\log
                    \kappa_n|}
                  \int_{\rho_n^{-1}(\Omega-a_n)}\left|\left(R_n
                      \left(\frac{2x}{1+|x|^2},0\right)\right)'\right|^2
                  \intd x.
		\end{split}  
	\end{align}
	Therefore, together with the fact that $|(R_nv)'| \leq |v|$ for
        all $v\in \R^3$ we have
	\begin{align}\label{eq:anisotropy_upper_bound}
		\begin{split}
                  \lim_{n\to \infty} \frac{r_0^2}{|\log \kappa_n|}
                  \int_{\rho_n^{-1}(\Omega-a_n)}\left|\left(R_n
                      \left(\frac{2x}{1+|x|^2},0\right)\right)'\right|^2
                  \intd x \qquad \qquad \\
                  \leq \liminf_{n\to \infty} \frac{8 \pi r_0^2}{|\log
                    \kappa_n|} \left( \int_0^1 s^3 \intd s +
                    \int_1^{\rho_n^{-1}\operatorname{diam}(\Omega)}
                    \frac{\intd s}{s} \right) = 8 \pi r_0^2,
		\end{split}
	\end{align}
        where we recalled that $\rho_n / \kappa_n \to r_0$ as
          $n \to \infty$ by the BP-convergence.
	
	At the same time, as
        $\lim_{n\to \infty} a_n = a_0 \in \Omega$, there exists
        an $\tilde s>0$ such that
        $B_{\tilde s} (a_n) \subset \Omega$ for all
        $n \in \mathbb{N}$.  By the estimate
        \eqref{eq:anisotropy_upper_bound} the expression on the
        right-hand side of estimate \eqref{eq:reduced_expression} is
        bounded, so that we can pass to the limit in the rotation
        $\lim_{n\to \infty} R_n = R_0$ and exploit
        $R_0 e_3 = e_3$, to get
	\begin{align}\label{eq:anisotropy_lower_bound}
		\begin{split}
                  & \quad \lim_{n\to \infty} \frac{r_0^2}{|\log
                    \kappa_n|}
                  \int_{\rho_n^{-1}(\Omega-a_n)}\left|\left(R_n
                      \left(\frac{2x}{1+|x|^2},0\right)\right)'\right|^2
                  \intd x \\
                  & \geq \limsup_{n\to \infty} \frac{r_0^2}{|\log
                    \kappa_n|} \int_{B_{\rho_n^{-1}\tilde
                      s}(0)}\frac{4|x|^2}{(1+|x|^2)^2} \intd x \\
                  & = \limsup_{n\to \infty} \frac{8 \pi r_0^2}{|\log
                    \kappa_n|} \int_1^{\rho_n^{-1}\tilde
                    s}\frac{\intd s}{s}  \\
                  & = 8\pi r_0^2.
		\end{split}
	\end{align}
        The statement then follows from combining the estimates
        \eqref{eq:anisotropy_upper_bound} and
        \eqref{eq:anisotropy_lower_bound}.
\end{proof}

\printbibliography

\end{document}